\documentclass[a4paper,11pt]{amsart}
\usepackage[british]{babel}
\usepackage[colorlinks=true,citecolor=black,urlcolor=black,linkcolor=black]{hyperref}

\newenvironment{quoting}{%
    \list{}{\leftmargin 1.6em \rightmargin\leftmargin}
    \item\relax} {\endlist}

\newtheorem{theorem}{Theorem}
\newtheorem{lemma}{Lemma}[section]
\theoremstyle{remark}
\newtheorem*{remark}{Remark}
\newtheorem*{question}{Question}

\newcommand{\sT}{T} % Interval transformation
\newcommand{\abs}[1]{\left|#1\right|} % Absolute value
\newcommand{\id}{\operatorname{id}} % The identity
\newcommand{\cC}{\mathcal{C}} % Differentiable
\newcommand{\cM}{\mathcal{M}} % Operator
\newcommand{\cK}{\mathcal{K}} % Operator
\newcommand{\bC}{\mathbb{C}} % Complex numbers
\newcommand{\cL}{\mathcal{L}} % Transfer operator
\newcommand{\cI}{\mathcal{I}} % Union of subintervals
\newcommand{\bN}{\mathbb{N}} % Natural numbers
\newcommand{\norm}[1]{\left\|#1\right\|} % Norm
\newcommand{\fA}{\mathfrak{A}} % Banach space
\newcommand{\fB}{\mathfrak{B}} % Banach space
\newcommand{\fJ}{\mathfrak{J}} % Vector space
\newcommand{\fC}{\mathfrak{C}} % Vector space
\newcommand{\snorm}[1]{\left\|#1\right\|_{\fB}} % Our norm
\newcommand{\characteristic}[1]{\mathbf{1}_{#1}} % Characteristic
\newcommand{\SBV}{\operatorname{SBV}} % SBV
\newcommand{\BV}{\operatorname{BV}} % BV
\newcommand{\Spec}[1]{\operatorname{Spec}(#1)} % Spectrum
\newcommand{\FullSpec}[1]{\normalfont\operatorname{\underbar{Spec}}(#1)} % Full spectrum
\newcommand{\EssSpec}[1]{\operatorname{EssSpec}(#1)} % Essential
\newcommand{\FullEssSpec}[1]{\normalfont\operatorname{\underbar{EssSpec}}(#1)} % Full essential
\newcommand{\jump}[2]{J\left (#1,#2 \right)} % Jump
\newcommand{\cball}[1]{\{\abs{\lambda} \leq #1\}} % Closed ball
\newcommand{\dD}{\operatorname{D}} % Weak derivative
\newcommand{\apart}{\operatorname{D}_{\operatorname{a}}} % Absolutely continuous part
\newcommand{\jpart}{\operatorname{D}_{\operatorname{j}}} % Jump part

\title{Discontinuities cause essential spectrum}

\author[O.~Butterley]{Oliver Butterley}
\address{(Oliver Butterley) Dipartimento di Matematica, II Università di Roma ``Tor Vergata'' -- Via della Ricerca Scientifica 1 -- 00133 Roma -- Italy}
\email{butterley@mat.uniroma2.it}

\author[G.~Canestrari]{Giovanni Canestrari}
\address{(Giovanni Canestrari) Dipartimento di Matematica, II Università di Roma ``Tor Vergata'' -- Via della Ricerca Scientifica 1 -- 00133 Roma -- Italy}
\email{canestrari@mat.uniroma2.it}

\author[S.~Jain]{Sakshi Jain}
\address{(Sakshi Jain) Dipartimento di Matematica, II Università di Roma ``Tor Vergata'' -- Via della Ricerca Scientifica 1 -- 00133 Roma -- Italy}
\email{jain@mat.uniroma2.it}

\begin{document}

\begin{abstract}
    We study transfer operators associated to piecewise monotone interval transformations and show that the essential spectrum is large whenever the Banach space bounds \(L^\infty\) and the transformation fails to be Markov.
    Constructing a family of Banach spaces we show that the lower bound on the essential spectral radius is optimal. 
    Indeed, these Banach spaces realise an essential spectral radius as close as desired to the theoretical best possible case.
\end{abstract}

\maketitle
\thispagestyle{empty}

%%%%%%%%%%%%%%%%%%%%%%%%%%%%%%%%
\section{Introduction}%
\label{sec:intro}
%%%%%%%%%%%%%%%%%%%%%%%%%%%%%%%%

Describing the statistical properties of a hyperbolic dynamical system is an important question and one powerful approach to this problem is to study the associated transfer operator acting on a ``good'' choice of Banach space (see \cite{Liverani04} and references within).
In recent decades these techniques have seen massive development
and many systems were studied in this way with valuable results being obtained~\cite{Keller89}.
Results include statistical properties (invariant measures, CLT, LLT, decay of correlations, large deviations, zeta functions, etc.), stability of these properties under perturbations~\cite{KL99} and rigorous numerics~\cite{Liverani01}.
Different Banach spaces were used in various applications and often it was beneficial to engineer appropriate Banach spaces.
At present we have, a ``bestiary'' \cite{Smania21} of available spaces, each with benefits for different settings and purposes.

Arguably, particularly with physical applications in mind, discontinuities are natural.
For example, in billiard systems discontinuities arise from tangential collisions~\cite{CM06} and in Lorenz-like flows discontinuities arise from shearing close to the singularity~\cite{Viana00}.
A weight of evidence suggests that discontinuities don't worsen the statistical properties of the system even if they cause technical difficulties (e.g., \cite{AM16,BDL18}).

Let \(\cL : \fB \to \fB\) be a transfer operator associated to some dynamical system.
Since the Banach space must contain the observables of interest, we take the point of view that, at the very least, it includes \(\cC^{\infty}\).
Consequently, it also contains \(\cL^{n}(\cC^{\infty})\) for all \(n\in \bN\)
and so discontinuities of the underlying system force the Banach space to contain observables with discontinuities.
On the other hand, if the space is big, the spectrum can be huge and consequently useless (further details in the following section).
The aim is to choose a Banach space for which the essential spectrum is as small as possible.
Once this is done, one may proceed to determine the point spectrum or ``resonances'' of the system  (see, e.g., \cite{BKL22,FGL19} and references within).
Moreover the point spectrum is essentially independent of the choice of Banach space, it is inherent in the system studied~\cite{BT08}.

Our present focus is piecewise monotone interval transformations.
Classically these are studied using \(\BV\) and the essential spectral radius is equal to a value depending on the expansion~\cite{Keller84}.
However, for \(C^{\infty}\) systems without discontinuities, the essential spectral radius can be made as small as desired by choosing a Banach space of smoother observables~\cite{CI91} and so we ask ourselves:

\begin{question}
    Does there exist a Banach space for studying piecewise monotone interval transformations on which the essential spectrum is as small as desired?
\end{question}

Somewhat unfortunately, one success of this present work is to answer this question in the negative when the transformation fails to be Markov.\footnote{A piecewise monotone transformation \(\sT:[0,1] \to [0,1]\) is said to be \emph{Markov} if it admits a finite partition \(\Omega\) which additionally has the property that, for any \(\omega,\omega^{\prime}\in \Omega\), if \(\sT \omega \cap \omega^{\prime} \neq \emptyset\) then \(\sT \omega \supseteq \omega^{\prime}\).}
In Theorem~\ref{thm:lower}, we show the existence of essential spectrum caused by the discontinuities of the transformation and which cannot be avoided by choice of Banach space.
The key idea is to construct eigenvalues of the dual operator in a way that relies on the discontinuities. 
Consequently this component of the spectrum has a different origin compared to the essential spectrum studied previously~\cite{CI91} and which is due to limited differentiability.
This part of the work gives a lower estimate on the essential spectral radius which holds for any Banach space.
In Theorem~\ref{thm:example} we construct an example of a piecewise expanding interval transformation for which this theoretical lower estimate is strictly smaller than the essential spectral radius studying the same system on \(\BV\)~\cite{Keller84}.
Consequently the natural next question is:

\begin{question}
    Given a piecewise monotone interval transformation which fails to be Markov, does there exist a Banach space on which the essential spectral radius is smaller than for \(\BV\), closer to the theoretical best?
\end{question}

In Theorem~\ref{thm:custom} we answer this question in the affirmative by demonstrating how, for any given expanding interval transformation, it is possible to construct a Banach space so that the essential spectral radius is arbitrarily close to the theoretical best case we have established in Theorem~\ref{thm:lower}.

Intimately connected to the present work and of physical relevance is the task of understanding when the correlations of \(C^{\infty}\) observables admit a description in terms of resonances (see definition and discussion in \cite[\S 1]{FGL19}).
It is known (see \cite{Ruelle90,BT07,GL08,FGL19} and citations within) that all uniformly expanding and uniformly hyperbolic \(\cC^{\infty}\) maps (in any dimension) admit Ruelle resonances to arbitrary precision.
This includes Markov expanding systems and pseudo Anosov diffeomorphisms, which are \(\cC^{\infty}\) in the sense that, apart from a set of discontinuities/singularities, they are \(\cC^{\infty}\) and this set is both invariant under the dynamics and has zero measure.
Essentially all proofs of resonances rely on spectral considerations for the transfer operator. 
As such, the present work demonstrates that discontinuities are an obstacle to the precision with which the correlations of observables can be described by resonances.

%%%%%%%%%%%%%%%%%%%%%%%%%%%%%%%%
\section{Results}%
\label{sec:results}
%%%%%%%%%%%%%%%%%%%%%%%%%%%%%%%%

We say that \(\sT:[0,1] \to [0,1]\) is a \(\cC^{r}\) \emph{piecewise monotone transformation} if there exists a finite partition \(\Omega\) of a full measure subset of \([0,1]\) into open intervals such that \(\left.\sT\right|_{\omega}\) is strictly monotone and \(r\)-times differentiable with continuous derivative for each \(\omega \in \Omega\). 
Let \(\cI\) denote the disjoint union of \(\overline{\omega}\), \(\omega\in\Omega\)
\footnote{In the rest of the article,
    it will be convenient to see \(T\) as a transformation from \(\mathcal{I}\) to \(\mathcal{I}\) rather than
    from \(I\) to \(I\). For each \(\omega\in \Omega\), \(T\) has a unique \(\cC^{0}\) extension
    to \(\overline{\omega}\). In such a way, \(T\) can be considered as a transformation from \(\mathcal{I}\) to itself.}.
We suppose throughout that \(\Omega\) denotes the maximal such partition.
The \emph{transfer operator} is defined\footnote{If \(\varphi\) is constant, or if \(\varphi = 1/{\abs{\sT^{\prime}}}\) then \(\cL_{\varphi}\) is the operator associated to the measure of maximal entropy or the SRB measure, respectively (see e.g.,~\cite{Baladi00}).} pointwise for any \(h: [0,1] \to \bC\),
\begin{equation}
    \label{eq:L-defin}
    \cL_{\varphi} h = \sum_{\omega \in \Omega} (\varphi \cdot h) \circ {{\left.\sT\right|}_{\omega}^{-1}} \cdot \characteristic{\sT\omega}.
\end{equation}
Note that \(\cL_{\varphi}\) doesn't necessarily leave\footnote{\(h\in\cC^\infty(\cI)\) means that \(h:[0,1]\to \bC\) and that \(\left.h\right|_{\omega}\in\cC^\infty(\omega)\) for each \(\omega\in \Omega\).} \(\cC^\infty(\cI)\) invariant, particularly in the cases of present interest.
We say that \(a \in \mathcal{I}\) is a \emph{non-trivial} discontinuity of \(\sT\) if \(a \in \partial \omega\), \(\omega \in \Omega\),
and \({\{\sT^k a\}}_{k=1}^{\infty}\) is an infinite set.
For each such non-trivial discontinuity of \(\sT\) we consider the quantities
\[
    \begin{aligned}
        \Lambda^{\inf}({\sT,\varphi,a})
        &= \displaystyle\liminf_{n\to \infty}\textstyle \abs{\prod_{k=1}^{n}\varphi(\sT^k a)}^{\frac{1}{n}}, \\ 
        \Lambda^{\sup}({\sT,\varphi,a})
        &= \displaystyle\limsup_{n\to \infty}\textstyle \abs{\prod_{k=1}^{n}\varphi(\sT^k a)}^{\frac{1}{n}}.    
    \end{aligned}
\]
A given transformation might have more than one non-trivial discontinuity but, at most, a finite number.
We define \(\Lambda^{\inf}(\sT,\varphi)\) and \(\Lambda^{\sup}(\sT,\varphi)\) to be equal to the maximum of \(\Lambda^{\inf}({\sT,\varphi,a})\) and \(\Lambda^{\sup}({\sT,\varphi,a})\), over all non-trivial discontinuities, respectively.
If \(\sT\) has no non-trivial discontinuities, i.e., is Markov,\footnote{Markov implies the set of all images of discontinuities is a finite set and so there is no non-trivial discontinuity. On the other hand, no non-trivial discontinuities means the set of all images of discontinuities is a finite set and this set describes a Markov partition.} we set \(\Lambda^{\inf}(\sT,\varphi)\) and \(\Lambda^{\sup}(\sT,\varphi)\)  equal to zero.
Whenever it happens that \(\Lambda^{\inf}(\sT,\varphi) = \Lambda^{\sup}(\sT,\varphi)\), we set \(\Lambda^{\inf}(\sT,\varphi) = \Lambda^{\sup}(\sT,\varphi) = \Lambda(\sT,\varphi)\).

Our first result gives a lower bound for the essential spectral radius.

\begin{theorem}%
    \label{thm:lower}
    Let \(\cL_{\varphi}\) be the transfer operator associated to a \(\cC^{0}\)
    piecewise monotone transformation \(\sT:[0,1] \to [0,1]\) and let \(\varphi:[0,1]\to \bC\) be bounded\footnote{In particular, \(\varphi\) admits a continuous extension up to \(\overline{\omega}\), for all \(\omega \in \Omega\).} and \(\cC^{0}\) on each \(\omega \in \Omega\).
    Suppose that \((\fB,\snorm{\cdot})\) is a Banach space satisfying,
    \begin{enumerate}
        \item \(\cC^\infty(\cI) \subseteq \fB\);
        \item \(\cL_{\varphi}\) extends to a continuous operator on \(\fB\);
        \item \label{itm:infty}\({\norm{h}}_{L^\infty} \le \snorm{h}\) for all \(h\in \bigcup_{n=0}^{\infty} \cL^{n}_{\varphi}(\cC^{\infty}(\cI))\).
    \end{enumerate}
    If \(\sT\) has a non-trivial discontinuity \(a \in [0,1]\) then the essential spectral radius of \(\cL_{\varphi} : \fB \rightarrow \fB\) is not less than \(\Lambda^{\inf}(\sT,\varphi, a)\).
\end{theorem}

\noindent
The proof of this result\footnote{We prove the slightly stronger statement \(\FullEssSpec{\cL_{\varphi}} \supseteq \cball{\Lambda^{\inf}}\) (see Section~\ref{sec:ess-spec}).} is the content of Section~\ref{sec:lower} and takes advantage of some details related to the essential spectrum which are described in Section~\ref{sec:ess-spec}.

In particular, this theorem means that, when a system fails to be Markov, even increasing the regularity of observables by choosing suitable Banach spaces, it is not possible to obtain a precision better that \(\Lambda^{\inf}(T,\varphi)\) in a Ruelle resonance description of correlations.

A particular case of the theorem is when \(\varphi\) is a constant as happens for the transfer operator associated to the measure of maximal entropy.
In this case Theorem~\ref{thm:lower} gives a lower estimate for the essential spectral radius which is equal to the known~\cite{BK90} upper estimate for the essential spectral radius when using \(\BV\) (equal to the constant value of \(\varphi\)).
In other words, when studying the transfer operator associated to the measure of maximal entropy, if there is a non-trivial discontinuity, then it is impossible to get a smaller essential spectral radius than the one obtained by using \(\BV\).

The \(\beta\)-transformations defined as \(x \mapsto \beta x \mod 1\) for \(\beta \in (1,2]\) (see, e.g., \cite{CT12,Parry60}) fit the setting of this present study.
In this case, also for the transfer operator associated to the SRB measure, the weight \(\varphi\) is constant because the derivative is constant. 
Again we see that, for all \(\beta\) such that the transformation fails to be Markov, the essential spectral radius can't be smaller than \(\frac{1}{\beta}\) (the essential spectral radius in \(\BV\)).

In the situation where one studies the transfer operator acting on a Banach space which is large, e.g., \(L^{p}\) with \(1 \le p < \infty\), then it is possible to construct approximate eigenfunctions which are piecewise constant and hence show that the spectrum of the transfer operator is the entire disc of radius equal to the spectral radius \footnote{Indeed, for all \(\epsilon > 0\) and \(n \in \mathbb{N}\) there exists a set \(A\) such that the sets \(\{T^{-k}A\}_{k=0}^{n-1}\) are pairwise disjoint and their measure add up to \(1-\epsilon\) (this can be done using a version of the Rokhlin Lemma for non-invertible transformations). Then, using this one can define a function  which satisfies pointwise \(\cL g(x) = \lambda g(x)\) on \(\{T^{-k}A\}_{k=0}^{n-2}\) and has unit \(L^{p}\) norm, providing approximate eigenvalues.}.
Similar results hold when the space is large, even if not as large as \(L^{2}\) (see \cite[Remark 1.5]{Baladi00} where a disc of eigenvalues is shown to exist in the case of Hölder spaces).

In the case of \(\cC^{\infty}\) expanding circle transformations, with the weight for the transfer operator again given by \(\varphi = 1/\abs{\sT^{\prime}}\), Collet \& Isola~\cite{CI91} studied the essential spectrum of the transfer operator acting on the space of \(\cC^r\) observables.
They showed that the essential spectrum is a full disc which gets arbitrarily small as \(r\) increases.\footnote{If \(m\) denotes the invariant measure, the logarithm of the essential spectral radius on \( \cC^{r}\) is equal to \(\lim_{n\to\infty}\frac{1}{n}\log \int_{0}^{1} \abs{{(\sT^n)}^{\prime}}^{-r} dm\).}
(See~\cite{GL03} for the higher dimensional expanding case studied on \(\cC^{r+\alpha}\) spaces.)
Similarly, for \(\cC^{\infty}\) hyperbolic systems, the essential spectrum can be made as small as desired for a specific choice of the Banach space~\cite{GL08}.
In the case when the piecewise monotone interval transformation is Markov, it is possible to make the essential spectrum as small as desired (see e.g., \cite{BKL22}).
For this reason, the requirement in Theorem~\ref{thm:lower} that the transformation has a non-trivial discontinuity, is essential.
However, particularly as the set of transformations with non-trivial discontinuity is dense amongst the set of all piecewise monotone transformations,\footnote{If the transformation doesn't have a non-trivial discontinuity then the orbit of the discontinuity is pre-periodic. Iteratively making smaller and smaller perturbations guarantees that the orbit isn't pre-periodic and hence the transformation fails to be Markov.}
the present work fills in much of the gap that had remained on this topic.

Piecewise monotone interval transformations, the setting of this current work, have routinely been studied using \(\BV\)~\cite{LY73,Rychlik83,Keller84,Keller89,BK90} and generalized bounded variation~\cite{Keller85,Butterley13,Butterley14}.
Keller~\cite[p.184]{Keller84} demonstrated that previous estimates for the essential spectral radius of transfer operators acting on \(\BV\) are optimal, i.e., the essential spectral radius is equal to \(\lim_{n}\norm{ \smash{ \prod_{k=0}^{n-1} \varphi\circ \sT^k} }_{L^\infty}^{1/n}\) in the case when \(\varphi = 1/\abs{\sT^{\prime}}\).
The argument uses a sequence of indicator functions supported on smaller and smaller intervals in order to produce the lower bound on the essential spectral radius.
Theorem~\ref{thm:lower} shows that we can't have a smaller essential spectrum than we would obtain using \(\BV\) if the transformation has a non-trivial discontinuity and \(\norm{\smash{ \prod_{k=0}^{n-1} \varphi\circ \sT^k}}_{L^\infty}^{1/n}\) is comparable to \(\Lambda^{\inf}_{\sT,\varphi}\).
The first quantity is never smaller than the second quantity but this raises the question:
Is there a transformation where the first is strictly greater?
Specializing to the case where \(\varphi = \frac{1}{\abs{\sT'}}\), we answer this question as follows.

\begin{theorem}
    \label{thm:example}
    For any \(c \in (0,1)\), there exists a piecewise monotone interval transformation \(\sT: [0,1] \to [0,1]\) with a non-trivial discontinuity such that
    \[
        \lim_{n\to \infty}\norm{ \smash{\tfrac{1}{(\sT^n)'}} }_{L^\infty}^{1/n}
        - \Lambda({\sT,\tfrac{1}{\abs{\sT'}}})
        > c.
    \]
\end{theorem}
\noindent
Section~\ref{sec:example} is devoted to the explicit construction of these examples and hence the proof of above result.

This theorem tells us that there are transformations where the theoretical best case of Theorem~\ref{thm:lower} is substantially different from the essential spectral radius obtained using \(\BV\).
Hence it could be possible to make the essential spectral radius smaller by choosing a ``better'' Banach space.
By constructing specialized Banach spaces we are able to achieve exactly this as per the following result.

\begin{theorem}
    \label{thm:custom}
    Suppose that \(\sT:[0,1] \to [0,1]\) is a \(\cC^\infty\) piecewise monotone transformation (with associated partition \(\Omega\)) and that \(\varphi:[0,1]\to \bC\) admits a \(\cC^\infty\) extension to \(\overline{\omega}\) for each \(\omega \in \Omega\) and is uniformly bounded away from zero.
    Further suppose that \(\sT\) is uniformly expanding\footnote{In the sense that there exists \(C>0\), \(0 < \lambda < 1\) such that \(\abs{(\sT^{n})^{\prime}}\geq C\lambda^{-n}\) and all \(n\).}
    and \({\|\varphi \cdot T^{\prime}\|}_{L^{\infty}} < \infty\).

    Then, for all \(\tilde{\Lambda} > \Lambda^{\sup}(\sT,\varphi)\), there exists a Banach space \(\fB \subset L^{\infty}([0,1])\) such that,
    \begin{enumerate}
        \item \(\cC^\infty(\cI) \subseteq \fB\);
        \item \(\cL_{\varphi}\) extends to a continuous operator on \(\fB\);
        \item The essential spectral radius of \(\cL_{\varphi}: \fB \to \fB\) is at most \(\tilde{\Lambda}\).
    \end{enumerate}
\end{theorem}
\noindent
The proof of the above result is the content of Section~\ref{sec:custom} where we explicitly construct the family of Banach spaces.
The required distortion result is the content of Section~\ref{sec:distort}.

This theorem tells that, when \(\Lambda^{\sup}(\sT,\varphi) = \Lambda^{\inf}(\sT,\varphi)\), there exists a particular choice of Banach space such that the essential spectral radius is as close to the theoretical minimum as desired.
However we have to pay a price for this and the major downside is that the Banach space is very closely tied to the given system, even similar transformations couldn't be studied on the same Banach space which unfortunately would make delicate study of perturbations impossible (see, e.g., \cite{KL99}).

The present results don't say anything about the point spectrum and a high degree of smoothness of the system doesn't prevent the existence of plenty of it~\cite{KR04}.
Sometimes it is possible to use other arguments in order to identify the point spectrum~\cite{FGL19,BKL22} but control on the essential spectrum, as obtained here, is the key prerequisite to further investigation.

\begin{remark}[Fractional Sobolev spaces]
    An important class of Banach spaces which have been useful for expanding transformations with discontinuities are the ``fractional'' Sobolev spaces (see, e.g., \cite{Thomine11, AS20, Smania21}).
    They have the advantage that they allow the study of unbounded observables.
    On the other hand, assumption (\ref{itm:infty}) of Theorem~\ref{thm:lower} is not satisfied.
    Let \(p \in \mathbb{N}\), \(t\in (0,\frac{1}{p})\) and consider the factional Sobolev space \(H_p^t\) \cite{Thomine11}. 
    The norm of the characteristic function of an interval is bounded by a constant independent of the interval \cite[Corollary 3.7]{Strichartz1967}.
    For simplicity consider the case of  \(\sT: x \mapsto k x \mod 1\), \(k\in \{2,3,\ldots\}\) and the transfer operator associated to the SRB measure acting on \(H_p^t\).
    One may search for eigenfunctions of the form \({h_z = \sum_{\ell=0}^{\infty} z^\ell h\circ \sT^\ell}\) where \(h\) is in the kernel of the transfer operator (cf.\,\cite[Lemma 5]{CI91}, general transformations and transfer operators can be treated by modifying this same idea).
    If the sum converges then \(h_z\) is an eigenfunction associated to eigenvalue \(z\).
    We choose \(h\) to be the difference of the characteristic functions of two intervals and then  convergence follows from the bound on the \(H_p^t\) norm of characteristic functions. 
    Consequently the essential spectral radius is at least \(k^{-1}\) in this case. (If \(A\) is an interval, then \(\characteristic{A} \circ T^{\ell} = \characteristic{T^{-\ell}A}\) and so \(T^{-\ell}A\) is composed of at most \(k^\ell\) intervals.)
    In particular, the essential spectrum on \(H_p^t\) is not smaller than if we were to use \(\BV\).
\end{remark}

\begin{remark}[\(L^\infty\) bound]
    Assumption~(\ref{itm:infty}) of Theorem~\ref{thm:lower} is unfortunate and we suspect that it could be removed since, for all choices of Banach space which come to mind, the conclusion of the theorem holds.
    When the Banach space is sufficiently large the argument is to construct eigenfunctions or approximate eigenfunctions.
    When the Banach space is smaller, as is the case we work with, the dual space is bigger and the present work takes advantage of this to construct the eigenvectors of the dual.
\end{remark}

\begin{remark}[Zeta functions]
    The study of resonances is intimately connected to zeta functions. 
    Milnor \& Thurston \cite[p.558]{MT88} gave an example of an interval transformation for which they showed that the zeta function cannot be extended meromorphically outside a specified disc. 
    The same is known for \(\beta\)-transformations when they aren't Markov (stated in \cite[Theorem 2.4]{FLP94}, using essentially \cite{Takahashi73}).
    In general, if the transfer operator (acting on a relevant Banach space) has essential spectral radius contained in a disc of radius \(r\), then the zeta function admits a meromorphic extension to a disc of radius \(r^{-1}\) (see, e.g., \cite{HK84,BK90} and \cite[Part I]{Baladi18}). 
    If this connection between spectrum and zeta function were true for general Banach spaces (as suggested by a weight of evidence but currently an open problem as far as the authors are aware) then the two results cited above would imply, in those settings, lower bounds on the essential spectral radius for the transfer operator acting on any relevant Banach space.
\end{remark}

We would hope that results, similar to this work, can be obtained, both in the case of piecewise \(\cC^{\infty}\) expanding transformations in higher dimension and billiard maps (2D, hyperbolic with discontinuities).
This seems probable although in higher dimension there are additional complications related to the geometry and propagation of discontinuities.

%%%%%%%%%%%%%%%%%%%%%%%%%%%%%%%%
\section{Lower bound}%
\label{sec:lower}
%%%%%%%%%%%%%%%%%%%%%%%%%%%%%%%%

Let \(T:[0,1]\to[0,1]\), \(\varphi:[0,1]\to\bC\), \(\fB\supseteq \cC^{\infty}(\cI)\) and  \(\left\{a_k\right\}_{k=1}^{\infty}\) (orbit of the non-trivial discontinuity, i.e., \(a_{k} = T^{k}a\), where \(a \in \partial \omega\),  \(\omega \in \Omega\) and \({\{\sT^k a\}}_{k=1}^{\infty}\) is an infinite set) be fixed according to the assumptions of Theorem~\ref{thm:lower} for the rest of this section.
Additionally, for notational simplicity, we drop the subscript and write \(\cL\) for \(\cL_{\varphi}\).

Since the discontinuities of \(\sT\) are key it is convenient to introduce some additional notation for this.
We may assume that \(\Omega\) is the maximal partition on which \(\sT\) is continuous.
Let \(N=\#\Omega\) and let \(0=c_{0} < c_{1} < \cdots < c_{N} = 1\) be the points of discontinuity in the sense that \(\Omega = \left\{ (c_{\ell-1}, c_{\ell}) \right\}_{\ell=1}^{N}\).
Let \(\Gamma = \{c_0,\ldots,c_N\}\).
There are \(2N\) sequences of points of the form\footnote{Here and subsequently, we use the superscript \(+\)/\(-\) for the left and right limits in the sense that $g(a_{k}^{+})=\lim_{\epsilon \rightarrow 0}g(a_{k} + \epsilon)$ and $g(a_{k}^{-})=\lim_{\epsilon \rightarrow 0}g(a_{k} - \epsilon)$.} \(\left\{\sT^{k}(c_{\ell}^{-})\right\}_{k=0}^{\infty}\) or \(\left\{\sT^{k}(c_{\ell}^{+})\right\}_{k=0}^{\infty}\).
We say that these are the discontinuity orbits and let \(\Delta\) denote the set of all points in all these orbits.
Each one of these discontinuity orbits may be a finite or infinite set of points.
In the case that the transformation is Markov the set would be finite.
For convenience we use the convention that \(\sT^{-1}y = \{x: y=\sT(x^-)\} \cup \{x: y=\sT(x^+)\}\).

\begin{lemma}
    \label{lem:ak-start}
    Without loss of generality we may suppose that \({\{a_k\}}_{k=1}^{\infty}\) is such that \(\sT^{-n}a_k \cap \Gamma = \emptyset\) whenever \(0 \le n<k\) and that \(\sT^{-1}a_1 \cap \Gamma \neq \emptyset\).
\end{lemma}

\begin{proof}
    If this isn't the case, then there would exist another discontinuity orbit \({\{b_k\}}_{k=1}^{\infty}\) such that \(a_m = b_n\) for some \(m,n\in \bN\), \(m<n\).
    In this case we can choose \({\{b_k\}}_k\) instead of  \({\{a_k\}}_k\).
\end{proof}

\begin{lemma}
    \label{lem:disc-orbit}
    There exists \(k_0 \in \bN\) such that \(\sT^{-1}a_{k} \cap \Delta = \{a_{k-1}\}\)  for \(k \ge k_0\).
\end{lemma}

\begin{proof}
    First observe that \(a_k \notin \Gamma\) for every \(k\in\bN\) by Lemma~\ref{lem:ak-start}.
    We must show that there exists \(k_0 \in \bN\) such that \(\sT^{-1}a_{k} \cap \Delta = \{a_{k-1}\}\) whenever \(k \ge k_0\).
    If this is false, it means there is another discontinuity orbit \({\{b_k\}}_{k=1}^{\infty}\subset \Delta\) such that \(\sT b_{j} = a_{m}\) for some \(j \in \bN\), \(m \geq k_0\).
    If this happens we increase \(k_0\) to be equal to \(m+1\).
    This can't happen again with the same discontinuity orbit since the orbit \(\{b_k\}\) and \(\{a_k\}\) are the same after this point.
    Since there are only a finite number of discontinuity orbits we only need to increase \(k_0\) at most a finite number of times.
\end{proof}

The above two lemmas are sufficient for the purposes of this section but, in a later section, we will prove Lemma~\ref{lem:delta}, a more detailed yet similar result.
Let \(k_0 \in \bN\) be fixed for the remainder of this section as established in Lemma~\ref{lem:disc-orbit}.

For each \(n\in \bN\), let \(\Omega_{n}\) denote the maximal partition into open intervals of a full measure subset of \([0,1]\) such that  \(\left.\sT^{n}\right|_{\omega}\) is strictly monotone and continuous for each \(\omega \in \Omega\).
Additionally, let \(\varphi_{n} = \prod_{k=0}^{n-1} \varphi \circ \sT^{k}\).
Iterating the definition~\eqref{eq:L-defin} of the transfer operator we obtain, for all \(n\in\bN\),
\begin{equation}%
    \label{eq:L-iterate}
    \cL^{n} h
    =
    \sum_{\omega \in \Omega_n} (\varphi_{n} \cdot h) \circ {\left({\left.\sT^n\right|}_{\omega}\right)}^{-1} \cdot \characteristic{\sT^{n}\omega}.
\end{equation}

Since it was assumed that \(\cC^\infty(\cI) \subseteq \fB\) and \(\cL\) is continuous on \(\fB\) we know that the set $\bigcup_{n=0}^{\infty} \cL^{n}(\cC^{\infty}(\cI))$ is contained within \(\fB\).
Let \(\fA_{0} \subseteq \fB\) denote the span of this set and let  \(\fA\) denote the completion of \(\fA_{0}\) in the \(\snorm{\cdot}\)-norm.
Observe that \(\fA \subseteq \fB\) is a complete invariant subspace.
In particular \((\fA,\snorm{\cdot})\) is a Banach space.
It may happen that \(\fA = \fB\) but this can't be assumed in general.

\begin{lemma}
    \label{lem:pw-cont}
    Any $g\in \fA_{0}$ is continuous except for at a finite set of points in \(\Delta \cup \Gamma\).
    Furthermore, for any $k\in \mathbb{N}$, $g(a_{k}^{+})$ and $g(a_{k}^{-})$
    are finite.
\end{lemma}

\begin{proof}
    We consider \(g=\cL^n h\) for some \(h\in\cC^\infty(\cI)\), \(n\in \bN\).
    This means~\eqref{eq:L-iterate} that
    \[
        g =
        \sum_{\omega \in \Omega_n} (\varphi_n \cdot h) \circ {\left({\left.\sT^n\right|}_{\omega}\right)}^{-1} \cdot \characteristic{\sT^{n}\omega}.
    \]
    Each summand is continuous on the interval \(\sT^{n}\omega\) and zero elsewhere.
    That it is only continuous is because \(\sT\) is not required to be better.
    To complete the proof we must show that the end points of each interval \(\sT^{n}\omega\) lie in \(\Delta\).
    Observe that each \(\omega \in \Omega_n\) is of the form
    \[
        \omega = \omega_{i_{0}} \cap \sT^{-1}(\omega_{i_{1}}) \cap ... \cap \sT^{-(n-1)}(\omega_{i_{n-1}}).
    \]
    where \(\Omega = \{\omega_j\}_{j=1}^{N}\) and \((i_{0},...,i_{n-1}) \in {\{1,\ldots,N\}}^{n}\).
    From this we see that a point can be an end point of the interval \(\sT^{n}\omega\) only if it is an end point of some \(\sT^k\omega_j\) where \(0\leq k\leq n\) and \(j\in  {\{1,\ldots,N\}}\).
    This set of points is contained in the set \(\Delta\) by definition.
    The continuity and boundedness of \(\varphi\) and \(h\), means that the left and right limits, $g(a_{k}^{+})$ and $g(a_{k}^{-})$, are finite.
\end{proof}

As a convenience, for any \(h:[0,1]\to \bC\) and \(x\in [0,1]\),
let 
\begin{equation}
    \label{eq:def-jump}
    \jump{h}{x} = h(x^{+}) - h(x^{-}),
\end{equation}
i.e., this is the ``jump'' in \(h\) at the point \(x\).
Lemma~\ref{lem:pw-cont} implies that \(\jump{h}{x}\) is defined for all \(h\in \fA_0\), \(x\in [0,1]\) and is zero when \(x\notin \Delta\).
Since \(\sT\) is not required to be orientation preserving we need a way to keep track of this behaviour.
For each \(k\), let
\begin{equation}
    \label{eq:def-sign}
    \gamma_k =
    \begin{cases}
        1  & \text{if \(\sT\) is increasing at \(a_k\)}, \\
        -1 & \text{if \(\sT\) is decreasing at \(a_k\)}.
    \end{cases}
\end{equation}
The following tells that the jumps behave, under the action of \(\cL\), like a shift.

\begin{lemma}
    \label{lem:calc-jump}
    If $g \in \fA_{0}$ and \(k\geq k_0\) then
    \[
        \jump{\cL g}{a_{k}} =  \gamma_{k-1} \ \varphi(a_{k-1}) \ \jump{g}{a_{k-1}}.
    \]
\end{lemma}

\begin{proof}
    In order to calculate \(\jump{\cL g}{a_{k}}\) we observe that \(\cL g\) is given by a sum of terms of the form
    \[
        (\varphi \cdot g) \circ {\left({\left.\sT\right|}_{\omega}\right)}^{-1} \cdot \characteristic{\sT\omega}.
    \]
    Because \(k\geq k_0\) we know that \(\sT^{-1}a_k \cap \Gamma = \emptyset\), in other words, none of the end points of the \(T\omega\) equal \(a_k\) (Lemma~\ref{lem:disc-orbit}).
    It therefore remains to consider discontinuities in \((\varphi \cdot g)\) at points in \(\sT^{-1}a_k\).
    Since the discontinuities of \(\varphi\) lie in \(\Gamma\), Lemma~\ref{lem:ak-start}, \(\varphi\) is continuous at points in \(\sT^{-1} a_k\).
    By Lemma~\ref{lem:pw-cont}, we know that \(g\) can only have discontinuities in \(\Delta\).
    Since \(\Delta \cap \sT^{-1}a_k = a_{k-1}\) this is the only term which contributes.
    The magnitude of the jump is equal to \(\varphi(a_{k-1}) \jump{g}{a_{k-1}}\).
    The sign depends on whether \({\left.\sT\right|}_{\omega}\) is orientation preserving or reversing for the \(\omega\) which contains \(a_{k-1}\) and so~\eqref{eq:def-sign} is given by \(\gamma_{k-1}\).
\end{proof}

In the subsequent argument it will be convenient to guarantee the existence of some element of \(\fA\) which has a non-zero jump at \(a_{k_0}\) but has no jump at \(a_k\) for all \(k>k_0\).
We can't just take \(\characteristic{[a_{k_0},1]}\) because we can't guarantee that such a function is in \(\fA_0\).
We can however take advantage of the dynamics to construct what we require.

\begin{lemma}
    \label{lem:h-K}
    There exists \(h_{\cK}\in \fA_0\) such that
    \(\jump{h_{\cK}}{a_{k_0}} = 1\)
    and, 
    whenever \(k>k_0\),
    \(\jump{h_{\cK}}{a_k} = 0\).
\end{lemma}

\begin{proof}
    If \(h_0 \in \cC^{\infty}(\cI)\) (to be chosen later) then~\eqref{eq:L-iterate},
    \[
        \cL^{k_0} h_0
        =
        \sum_{\omega \in \Omega_{k_0}} (\varphi_{k_0} \cdot h_0)\circ {\left({\left.\smash{\sT^{k_0}}\right|}_{\omega}\right)}^{-1} \cdot \characteristic{\sT^{k_0}\omega}
        \in \fA_0.
    \]
    Similar to the proof of Lemma~\ref{lem:pw-cont}, we see that a point is a discontinuity of \(\cL^{k_0} h_0\) only if it is an end point of the interval \(\sT^n \omega\) where \(0 \leq n \leq {k_0}\), \(\omega \in \Omega\).
    Lemma~\ref{lem:ak-start} tells us that \(\sT^{-n}a_{k} \cap \Gamma = \emptyset\) when \(n < k\) and so \(\jump{\cL^{k_0} h_{\cK}}{a_k} = 0\) whenever \(k>k_0\).
    On the other hand, \(\sT^{-1}a_1 \subseteq \sT^{-k_0}a_{k_0}\) and \(\sT^{-1}a_1 \cap \Gamma \neq \emptyset\) according to Lemma~\ref{lem:ak-start}.
    It is possible that \( \sT^{-k_0}a_{k_0} \cap \Gamma\) contains more than a single point.
    We therefore choose \(h_0 \in \cC^{\infty}(\cI)\) such that it is non-zero at just one point in this set.
    It is even possible that \( \sT^{-k_0}a_{k_0} \cap \Gamma\) contains the same point twice since it might exist as both the left and right limit.
    However this is also not a problem since \(h_0\in \cC^\infty(\cI)\) can be chosen to be non-zero at one side of a point in \(\Gamma\) and zero the other side.
    This all means that \(\jump{\cL^{k_0} h_0}{a_{k_0}} \neq 0\) and, scaling appropriately, we insure that \(\jump{\cL^{k_0} h_0}{a_{k_0}} = 1\).
\end{proof}

Recall that \(\Lambda^{\inf} = \liminf_{n\to \infty} \abs{ \varphi_{n}(a_{1}) }^{\frac{1}{n}}\).
We may assume that \(\varphi(a_k)>0\) for all \(k\in \bN\) because, if this were not the case, then \(\Lambda^{\inf} = 0\) and Theorem~\ref{thm:lower} is trivially true since the essential spectral radius is never less than zero. Hence, we may also assume that \(\Lambda^{\inf} > 0\).
Define the sequence \( \left\{\alpha_k\right\}\) by setting \(\alpha_{k_0-1}=1\) and by requiring that
\begin{equation}
    \label{eq:def-alpha}
    \alpha_{k}
    = \frac{\gamma_{k-1}}{\varphi(a_{k-1})} \alpha_{k-1}
    \quad \text{whenever \(k \geq k_0\)}.
\end{equation}
This definition implies that \(\abs{\alpha_{k}} = \abs{\varphi_{k_0-1}(a_1)} / \abs{\varphi_{k}(a_1)}\)
and, moreover, that \(\limsup_{k\to\infty} \abs{\alpha_k}^{1/k} = {1}/{\Lambda^{\inf}}\).
For any \(\abs{\lambda} < \Lambda^{\inf}\), define the linear functional \(\ell_{\lambda}: \fA_{0} \to \bC\),
\begin{equation}
    \label{eq:def-l-func}
    \ell_{\lambda}(h) =
    \sum_{k=k_0}^{\infty}  \lambda^{k} \alpha_k \, \jump{h}{a_k}.
\end{equation}

\begin{lemma}%
    \label{lem:good-funcs}
    If \(\abs{\lambda} < {\Lambda^{\inf}}\) then \(\ell_{\lambda}\) extends to a linear functional on \(\fA\).
\end{lemma}

\begin{proof}
    For any \(h\in \fA_{0}\) we know that \(\abs{\jump{h}{a_k}}\leq 2 \norm{h}_{L^{\infty}} \leq 2 \snorm{h}\) (here we use the assumption that \(\norm{\cdot}_{L^\infty} \leq \norm{\cdot}_{\fB}\)).
    Additionally, since \(\abs{\lambda} < \Lambda^{\inf}\),
    \(C_{\lambda} = \sum_{k=k_0}^{\infty} \abs{ \lambda^k \alpha_k }\) is finite.
    Consequently, \(\abs{\ell_{\lambda}(h)}\leq 2 C_{\lambda} \snorm{h}\) for all \(h\in \fA_{0}\) and so \(\ell_{\lambda}\) extends to \(\fA\).
\end{proof}

Let \(h_{\cK}\in \fA_0\), as constructed in Lemma~\ref{lem:h-K}, be fixed for the remainder of this section.

\begin{lemma}
    \label{lem:calc-l-hK}
    If \(\abs{\lambda} < {\Lambda^{\inf}}\) then \(\ell_{\lambda}(h_{\cK}) = \lambda^{k_0} \alpha_{k_0}\).
    In particular \(\ell_{\lambda}\) is non-zero as an element of \(\fA^{*}\), the dual of \(\fA\).
\end{lemma}

\begin{proof}
    Lemma~\ref{lem:h-K} tells that \(\jump{h_{\cK}}{a_k} = 0\) whenever \(k>k_0\), hence
    \[
        \ell_{\lambda}(h_{\cK})
        =  \sum_{k=k_0}^{\infty} \lambda^k \alpha_k     \, \jump{h_{\cK}}{a_k}
        =  \lambda^{k_0} \alpha_{k_0} \jump{h_{\cK}}{a_{k_0}}.
    \]
    This means that \(\ell_{\lambda}(h_{\cK}) = \lambda^{k_0} \alpha_{k_0}\) since, by Lemma~\ref{lem:h-K}, \(\jump{h_{\cK}}{a_{k_0}}=1\).
\end{proof}

Using \(h_{\cK}\in \fA_0\) we define the operator \(\cK : \fA_{0} \to \fA_{0}\),
\begin{equation}%
    \label{eq:def-K}
    \cK h = \alpha_{k_0}^{-1}\jump{h}{a_{k_0-1}} \, h_{\cK}.
\end{equation}
Observe that, for all \(h\in \fA_{0}\), \(  \snorm{\cK h} \leq 2 \abs{\smash{\alpha_{k_0}}}^{-1}  \norm{h}_{L^{\infty}} \snorm{\smash{h_{\cK}}} \le  C \snorm{h}\).
Consequently \(\cK\) admits a continuous extension to \(\fA\) which, abusing notation we also denote \(\cK\).
Observe that this is a rank-one operator.
Now let
\begin{equation}
    \label{eq:def-M}
    \cM :  \fA \to \fA; \quad
    \cM  = {\left.\cL\right|}_{\fA} - \cK.
\end{equation}
Since \(\cM\) and \(\cL\) differ by a compact operator they will have the same essential spectral radius.
We now show that the linear functionals \(\ell_{\lambda}\) are eigenvectors for the dual operator \(\cM^{*}\).

\begin{lemma}%
    \label{lem:eigenvecs}
    If \(\abs{\lambda} < {\Lambda^{\inf}}\) then \(\cM^{*} \ell_{\lambda} = \lambda \ell_{\lambda}\) as elements of \(\fA^{*}\).
\end{lemma}

\begin{proof}
    Let \(h\in \fA_{0}\).
    We will calculate separately the two terms of \( \ell_{\lambda}(\cM h)   =  \ell_{\lambda}(\cL h)     - \ell_{\lambda}(\cK h)\).
    For the \(\cL\) term,
    \begin{equation*}
        \ell_{\lambda}(\cL h)
        =
        \sum_{k=k_0}^{\infty}  \lambda^{k} \alpha_k \, \jump{\cL h}{a_k}.
    \end{equation*}
    By definition~\eqref{eq:def-alpha}, \(\alpha_k \gamma_{k-1}\varphi(a_{k-1}) = \alpha_{k-1}\), 
    by Lemma~\ref{lem:calc-jump},
    \(\jump{\cL g}{a_{k}} =  \gamma_{k-1} \ \varphi(a_{k-1}) \ \jump{g}{a_{k-1}}\)
    and so \(\lambda^k \alpha_k \, \jump{\cL h}{a_k} = \lambda \left[\lambda^{k-1}\alpha_{k-1} \, \jump{h}{a_{k-1}}\right] \) for each \(k\geq k_0\).
    Consequently,
    \begin{equation}
        \label{eq:calc-L}
        \begin{aligned}
            \ell_{\lambda}(\cL h)
             & = \lambda \sum_{k=k_0-1}^{\infty}  \lambda^{k} \alpha_k \, \jump{ h}{a_k} \\
             & = \lambda^{k_0} \jump{h}{a_{k_0-1}}
            + \lambda \ell_{\lambda}(h).
        \end{aligned}
    \end{equation}
    For the \(\cK\) term, recall that \(\cK h = \alpha_{k_0}^{-1} \jump{h}{a_{k_0-1}} \, h_{\cK}\) and so
    \( \ell_{\lambda}(\cK h) = \alpha_{k_0}^{-1} \jump{h}{a_{k_0-1}} \, \ell_{\lambda}(h_{\cK})\).
    Lemma~\ref{lem:calc-l-hK} tells us that \(\ell_{\lambda}(h_{\cK}) = \lambda^{k_0} \alpha_{k_0}\) and so
    \begin{equation}
        \label{eq:calc-K}
        \ell_{\lambda}(\cK h)
        = \lambda^{k_0} \jump{h}{a_{k_0-1}}.
    \end{equation}
    Summing the terms from the above calculations, \eqref{eq:calc-L} and \eqref{eq:calc-K}, we have shown that \( \ell_{\lambda}(\cM h) = \ell_{\lambda}(\cL h) - \ell_{\lambda}(\cK h) =  \lambda \ell_{\lambda}(h)\) for any \(h \in \fA_{0}\).
    By continuity this holds also on \(\fA\).
\end{proof}

We denote the spectrum and the essential spectrum of an operator by  \(\Spec{\cdot}\) and \(\EssSpec{\cdot}\) respectively, and  their respective ``filled'' versions are denoted  by \(\FullEssSpec{\cdot}\) and \(\FullSpec{\cdot}\) (see Section~\ref{sec:ess-spec} for the precise definitions).
To complete the proof it is convenient to recall the different spaces and operators we worked with.
The Banach space \((\fA,\snorm{\cdot})\) is an \(\cL\)-invariant subspace of the Banach space \((\fB,\snorm{\cdot})\).
We worked with the operators
\[
    \cL : \fB \to \fB, \quad
    \cM : \fA \to \fA, \quad
    \cK : \fA \to \fA,
\]
and we know that, by definition of \(\cM\),
\[
    \left.\cL\right|_{\fA} = \cM + \cK.
\]
Lemma~\ref{lem:calc-l-hK} tells us that \(\ell_\lambda \in \fA^*\) is non-zero and so
Lemma~\ref{lem:eigenvecs} implies that, whenever \(\abs{\lambda}<\Lambda^{\inf}\), \(\lambda\) is an eigenvalue of \(\cM^{*}: \fA^{*} \to \fA^{*}\).
Moreover the spectrum is a closed set and so%
\[
    \Spec{\cM} = \Spec{\cM^{*}} \supseteq \cball{\Lambda^{\inf}}.
\]
Lemma~\ref{lem:U-far-out} implies that \(\FullEssSpec{\cM} \supseteq \cball{\Lambda^{\inf}}\).
Since \(\cK\) is compact,
\[
    \EssSpec{\cM} = \EssSpec{\left.\cL\right|_{\fA}}
\]
and so \(\FullEssSpec{\left.\cL\right|_{\fA}} \supseteq \cball{\Lambda^{\inf}}\) 
and, in particular, \(\FullSpec{\left.\cL\right|_{\fA}} \supseteq \cball{\Lambda^{\inf}}\).
To finish the argument we need to compare the spectrum of the restriction \(\left.\cL\right|_{\fA}\) with the spectrum of \(\cL\).
By Lemma~\ref{lem:restriction}, \(\Spec{\left.\cL\right|_{\fA}} \subseteq \FullSpec{\cL}\), and so we know that \(\FullSpec{\cL} \supseteq \cball{\Lambda^{\inf}}\).
This implies (Lemma~\ref{lem:U-far-out}), that
\[
    \FullEssSpec{\cL} \supseteq \cball{\Lambda^{\inf}}
\]
and, in particular, the essential spectral radius of \(\cL\) is not smaller than \(\Lambda^{\inf}\).
This concludes the proof of Theorem~\ref{thm:lower}.

%%%%%%%%%%%%%%%%%%%%%%%%%%%%%%%%
\section{Background on essential spectrum}
\label{sec:ess-spec}
%%%%%%%%%%%%%%%%%%%%%%%%%%%%%%%%

In this section we isolate some relatively standard functional analytic details which are useful for the present work.
We say that \(z\) lies in the essential spectrum \(\EssSpec{\cL}\) of a bounded operator \(\cL\) if \((z\id - \cL)\) is not a Fredholm operator.
There are five different definitions of essential spectrum in common usage (see, e.g., \cite[I.4]{EE87} where the one we use corresponds to the third).
The different definitions in general aren't equivalent however the essential spectral radius is equal for all five definitions~\cite[Corollary 4.11]{EE87}.
The notion of essential spectrum we use here is invariant under compact perturbations (this is true for four out or the five definitions)~\cite[Theorem 4.1]{EE87}.

Let \(\FullEssSpec{\cL}\) denote the complement of the unbounded component of \(\bC \setminus \EssSpec{\cL}\).
I.e., this is the union of \(\EssSpec{\cL}\) and all the bounded components of the complement.
Similarly \(\FullSpec{\cL}\) denotes the complement of the unbounded component of \(\bC \setminus \Spec{\cL}\).
These are the ``filled'' versions of the usual notions.

In our work we arrive at the point where we know that the spectrum contains a complete disc.
According to one of the definitions of the essential spectrum this would imply that the entire disc is essential spectrum.
However this version of the essential spectrum might not be invariant under compact perturbation and this is a property which we require in the present argument.
We know that~\cite[Theorem 4.3.18]{Davies07}:
\begin{quoting}
    If \(\cL\) is a bounded operator then \(\Spec{\cL} \setminus \FullEssSpec{\cL}\) consists of isolated eigenvalues with finite algebraic and geometric multiplicities.
\end{quoting}
This allows us to conclude the following observation.

\begin{lemma}%
    \label{lem:U-far-out}
    Let \(\cL\) be a bounded operator
    and let \(C>0\).
    If \(\FullSpec{\cL} \supseteq \cball{C}\)
    then \(\FullEssSpec{\cL} \supseteq \cball{C}\).
\end{lemma}

\begin{proof}
    Since \(\FullSpec{\cL} \supseteq \cball{C}\) there must exist a connected set \(S \subset \bC\) such that  \(S \supseteq \cball{C}\) and\footnote{We denote by \(\partial S\) the topological boundary of \(S\subset \bC\).} \(\partial S \subset \Spec{\cL}\).
    Let \(\lambda\in \partial S\) and, for sake of contradiction, suppose that \(\lambda \notin \FullEssSpec{\cL}\).
    This means~\cite[Theorem 4.3.18]{Davies07} that \(\lambda\) is an isolated eigenvalue which is in contradiction to the fact that \(\lambda\) is in the boundary of \(S\).
    Consequently \(\partial S \subset \FullEssSpec{\cL}\) and so \(S \subseteq \FullEssSpec{\cL}\).
\end{proof}

In general we couldn't hope for better than the above inclusion (the result requires \(\FullEssSpec{\cL}\) rather than \(\EssSpec{\cL}\)) because there exist examples\footnote{%
    Consider \(\cL:l^{2}(\mathbb {Z} )\to l^{2}(\mathbb {Z} ) \) defined as \((a_1,a_2,a_3,\ldots) \mapsto (0, a_1,a_2,\ldots) \).
    If \(\abs{z}=1\) then the range of \((z\id - \cL)\) is dense but not closed and so this set is included in every definition of essential spectrum.
    If \(\abs{z}<1\) then \((z\id - \cL)\) has a closed range, one-dimensional kernel, and one-dimensional cokernel and so this set is in \(\Spec{\cL}\) but not in \(\EssSpec{\cL}\).}
where the entire unit disc is in the spectrum but that only the unit circle is in the part of the spectrum which is invariant under compact perturbation.

We also need to be able to connect the the essential spectrum of the restriction of an operator to the spectrum of the operator itself.
It would be convenient if the spectrum of the restriction was always smaller but this isn't true in general.
Nonetheless this type of relationship holds for the part of the spectrum which we are interested in.
The following lemma tells us that the spectrum of \(\left.\cL\right|_{\fA}\) can be richer than the spectrum of \(\cL\) only if it ``fills'' holes present in the spectrum of \(\cL\).

\begin{lemma}%
    \label{lem:restriction}
    Let \(\cL\) be a bounded operator on Banach space \((\fB,\snorm{\cdot})\).
    Suppose that \(\fA \subseteq \fB\) is a closed \(\cL\)-invariant subspace.
    Then
    \begin{equation}
        \label{eq:spec-res}
        \Spec{\left.\cL\right|_{\fA}} \subseteq \FullSpec{\cL}.
    \end{equation}
\end{lemma}

\begin{proof}
    We know\footnote{Let \(\lambda \in \partial \Spec{\cL}\). Choose sequence \(\lambda_k \to \lambda\) such that \(\lambda_k \notin \Spec{\cL}\) for all \(k\in \bN\). Since \(\norm{(\lambda_k - \cL)^{-1}}\to \infty\) there exists \({\{h_k\}}_{k}\) such that \(\norm{h_k}\to 0\) but \(\norm{(\lambda_k - \cL)^{-1}h_k}= 1\) for each \(k\). Let \(g_k = (\lambda_k - \cL)^{-1}h_k\), these are the approximate eigenfunctions.} that  \(\partial \Spec{\left.\cL\right|_{\fA}}\) is contained in the approximate point spectrum of \({\left.\cL\right|_{\fA}}\).
    The approximate point spectrum of \(\left.\cL\right|_{\fA}\) is contained in the approximate point spectrum of \(\cL\), in particular in \(\Spec{\cL}\).
    In summary, we have \(\partial \Spec{\left.\cL\right|_{\fA}} \subseteq \Spec{\cL}\).
    Consequently \(\Spec{\left.\cL\right|_{\fA}} \subseteq \FullSpec{\cL}\).
\end{proof}

%%%%%%%%%%%%%%%%%%%%
\section{Examples}
\label{sec:example}
%%%%%%%%%%%%%%%%%%%%

This section is devoted to proving Theorem~\ref{thm:example} by explicitly constructing a family of examples.

For any \(m \in \bN\), \(m \ge 4\) and \(\rho \in (1,m)\) we define the piecewise monotone (indeed affine) transformation \(\sT_{m,\rho}: [0,1] \to [0,1]\),
\[
    \sT_{m, \rho}(x) =
    \begin{cases}
        \frac{m}{m-3}x,        & \mbox{if }x\in [0,\frac{m-3}{m})= \omega_{1}             \\
        m(x-\frac{m-3}{m}),    & \mbox{if }x\in[\frac{m-3}{m},\frac{m-2}{m})= \omega_{2}  \\
        m(x-\frac{m-2}{m}),    & \mbox{if }x\in[\frac{m-2}{m},\frac{m-1}{m}) = \omega_{3} \\
        \rho(x-\frac{m-1}{m}), & \mbox{if }x\in[\frac{m-1}{m},1] = \omega_{4}.            \\
    \end{cases}
\]
Observe that \(\sT_{m, \rho}\), restricted to \(\omega_1\), \(\omega_2\) or \(\omega_3\) is bijective onto \([0,1)\),
while \(T_{m,\rho}(\omega_{4}) = [0, \frac{\rho}{m}]\).

\begin{lemma}
    \label{lem:ourRadius}
    For any \(m \ge 4\), there exists \(\rho \in (0,m)\) such that \(1\) is a non-trivial discontinuity for the transformation \(\sT_{m,\rho}\)
    and\footnote{Recall that we write \(\Lambda\) when \(\Lambda^{\inf} = \Lambda^{\sup}\), i.e. when the limit exists.}
    \[
        \Lambda(\sT_{m,\rho}, \tfrac{1}{\sT_{m,\rho}^\prime}) = \frac{1}{m}.
    \]
\end{lemma}

\begin{proof}
    Notice that \(T_{m, \rho}\) restricted to \(\omega_{2}\cup \omega_{3}\) is independent of \(\rho\).
    Furthermore, \(T_{m, \rho}\) \(\omega_{2} \supset \omega_{1}\cup \omega_{2}\) and
    \(T_{m, \rho}\) \(\omega_{1} \supset \omega_{1}\cup \omega_{2}\).
    Therefore, there exists a point \(b \in \omega_{2} \cup \omega_{3}\) (actually a full Cantor set of points),
    which remains forever in \(\omega_{2}\cup \omega_{3}\) and enjoys a non-periodic trajectory.
    Let us fix \(\rho\) in such a way that
    \[
        \sT_{m,\rho}(1) = \frac{\rho}{m} = b .
    \]
    By the non-periodicity of the orbit of \(b\), we have that \(1\) is a non-trivial discontinuity for the map \(T_{m,\rho}\).
    Furthermore, since such an orbit will visit only the sets \(\omega_{2}\) and \(\omega_{3}\), we have that
    \[
        \Lambda(\sT_{m,\rho}, \tfrac{1}{\sT_{m,\rho}^\prime}, 1)
        = \lim_{n\to \infty} \biggl ( \frac{1}{\rho}\cdot (1/m)^{n-1}\biggr )^{1/n}
        = \frac{1}{m}.
    \]
    To finish, we observe that \(1\) is the only non-trivial discontinuity so that
    \(\Lambda(\sT_{m,\rho}, \tfrac{1}{\sT_{m,\rho}^\prime}) = \Lambda(\sT_{m,\rho}, \tfrac{1}{\sT_{m,\rho}^\prime}, 1) \).
\end{proof}

On the other hand, the next lemma gives an estimate of the quantity which corresponds to the essential spectral radius on BV.
\begin{lemma}
    \label{lem:KellRadius}
    For all \(m \ge 4\) and \(\rho \in (0,m)\),
    \[
        \lim_{n\to \infty}\norm{ {\tfrac{1}{(\sT_{m,\rho}^n)'}} }_{L^\infty}^{1/n}
        \ge \frac{m-3}{m}.
    \]
\end{lemma}

\begin{proof}
    We take advantage of the fact that \(0\) is a fixed point.  For every \(n\in \bN\), let \(I_{n} = [0,\frac{1}{10}(\frac{m-3}{m})^{n}]\).
    We observe that \( \norm{{1}/{(T^n)'}}_{L^\infty([0,1])} \ge  \norm{{1}/{(T^n)'}}_{L^\infty(I_n)}\).
    On the other hand, for every \(m \ge 4\), \(\norm{{1}/{(T^n)'}}_{L^\infty(I_n)} = (\frac{m-3}{m})^{n}\), since \(T^{k}(x) \in \omega_{1}\) for every \(k \le n\).
\end{proof}

Let \(c\in (0,1)\) and choose \(m\) sufficiently large such that
\[
    \frac{m-3}{m} - \frac{1}{m} > c.
\]
Then, by Lemma \ref{lem:ourRadius}, we can fix \(\rho\) such that \(1\) is a non-trivial discontinuity for the transformation \(T_{m,\rho}\) and \(\Lambda(\sT_{m,\rho}, \tfrac{1}{\sT_{m,\rho}^\prime}) = \frac{1}{m}\).
On the other hand, thanks to Lemma~\ref{lem:KellRadius}, \(        \lim_{n\to \infty}\norm{ {{1}/{(\sT_{m,\rho}^n)'}} }_{L^\infty}^{1/n}  \ge \frac{m-3}{m}\), independent of \(\rho\).
This concludes the proof of Theorem~\ref{thm:example}.

%%%%%%%%%%%%%%%%%%%%%%%%%%%%%%%%
\section{Custom Banach space}%
\label{sec:custom}
%%%%%%%%%%%%%%%%%%%%%%%%%%%%%%%%

This section is devoted to the construction of custom Banach spaces,
and hence the proof of Theorem \ref{thm:custom}.

The argument of Keller~\cite{Keller84} for the lower bound to the essential spectral radius on \(\BV\) suggests that we need to be more restrictive on the possibility of jumps in the observables.
A first choice could be \(\SBV\), the space of ``special'' bounded variation functions.
However this space is still too large in the sense that it still permits the construction that Keller uses.
Learning from this we permit discontinuities only where we absolutely need to because
of the discontinuities of the interval transformation. This permits us to encode the asymptotic behaviour of the
discontinuities introducing norms adapted to their orbits.

Throughout this section we work under the assumptions of Theorem \ref{thm:custom}, in particular we suppose that \(\sT:[0,1] \to [0,1]\) is a \(\cC^\infty\) piecewise monotone transformation (with associated partition \(\Omega\)) and that \(\varphi:[0,1]\to \bC\) has a \(\cC^\infty\) extension to \(\overline{\omega}\) for each \(\omega \in \Omega\) and is uniformly bounded away from zero.
Moreover \(\sT\) is uniformly expanding.
Let \(\eta = \lim_{n\to \infty}\norm{\smash{\prod_{k=0}^{n-1}\varphi(\sT^k)}}_{L^\infty}^{1/n}\)
and \(\lambda = \lim_{n\to \infty} \norm{1/(\sT^n)'}_{L^\infty}^{1/n} < 1\).
Fix, for the remainder of this section, \(\widetilde{\Lambda}>\Lambda^{\sup}(T, \varphi)\).
Since \(\lambda<1\) we may choose \(r\in \bN\), \(r\geq 1\) sufficiently large such that \(\eta \lambda^{r-1} < \widetilde{\Lambda}\).

\begin{remark}
    The assumptions of Theorem~\ref{thm:custom} can be relaxed to \(\sT\) and \(\varphi\) being merely \(\cC^r\) as long as \(r\in \bN\) is sufficiently large such that the above inequality is satisfied.
\end{remark}

Recall that, as defined in Section~\ref{sec:lower}, \(\Gamma \subset \cI\) denotes the boundary of the partition elements \(\omega \in \Omega\) and \(\Delta\) denotes \(\cup_{k=0}^{\infty}\sT^k\Gamma\), the set of all images of all discontinuities.
The following is an extension and rephrasing of Lemma~\ref{lem:ak-start} and Lemma~\ref{lem:disc-orbit}.

\begin{lemma}
    \label{lem:delta}
    There exists a finite set of points \({\{b_k\}}_{k=1}^{N_2}\) and finite number (indexed by \(1\leq j \leq N_1\)) of countable sets of points \({\{a_{j,k}\}}_{k=0}^{\infty}\) such that,
    \(a_{j,k} \notin \Gamma\),
    \[
        \Delta = {\{b_k\}}_{k=1}^{N_2} \cup \bigcup_{j=1}^{N_1} {\{a_{j,k}\}}_{k=0}^{\infty},
    \]
    and,
    \begin{equation}
        \label{eq:good-a}
        \sT^{-n}a_{j,k} \cap \Delta = \{a_{j,k-n}\},
        \quad \quad
        \text{ for all \( j \), \(k \geq n\)}.
    \end{equation}
\end{lemma}

\begin{proof}
    As mentioned at the beginning of Section~\ref{sec:lower}, there are at most a finite number of non-trivial discontinuity orbits and this is the origin of \(N_1\).
    Lemma~\ref{lem:disc-orbit} implies the result stated here, with the minor difference that now we keep track of all discontinuity orbits.
    For each of these infinite orbits there exists some iterate after which the required property \eqref{eq:good-a} is satisfied and the orbit will not return to \(\Gamma\) (Lemma~\ref{lem:ak-start} \& Lemma~\ref{lem:disc-orbit}).
    In this way we define the \(a_{j,k}\).
    The \(b_k\) are defined to be the finite number of points of \(\Delta\) remaining.
\end{proof}

For the remainder of this section we suppose that the \(b_k\) and \(a_{j,k}\) are fixed as determined in Lemma~\ref{lem:delta}.
We write \(\partial\Omega_{n} = \cup_{\omega \in \Omega_n} \partial \omega\).

\begin{lemma}
    \label{lem:disc-of-B} 
    \(T^{-n}(a_{j,k+n}) \cap \partial \Omega_{n} = \emptyset \)
    for all \(j,k,n\).
    In particular \(a_{j,k} \notin \partial \Omega_{n}\).
\end{lemma}

\begin{proof}
    Suppose, for sake of contradiction, there exists \(y\) such that \(T^{n}y = a_{j,k+n}\) and \(T^{p}y \in \Gamma\) for \(p \le n-1\).
    Then, \(T^{n-p}T^{p}y = a_{j,k+n}\) and so there exists \(a \in \Gamma\) such that  \(\sT^{n-p} a = a_{j,k+n}\).
    This is a contradiction by Lemma~\ref{lem:delta}.
    The final claim follows from the fact that \(a_{j,k} \in T^{-n}(a_{j,k+n})\).
\end{proof}

We say that  \(h: [0,1] \to \bC\) is piecewise \(\cC^{\infty}\) with respect to a finite partition \(0=p_{0} < p_{1}<... <p_{N} =1\),
if \({\left.h\right|}_{(p_{l}, p_{l+1})}\) admits a \(\cC^{\infty}\) extension
up to \([p_{l}, p_{l+1}]\), for all \(0 \le l \le N-1\).
We define \(\fB_0\) to be the set of all such \(h\) for which there exists a finite partition of \([0,1]\) with \(p_l \in \Delta\),
which makes \(h\) piecewise \(\cC^{\infty}\) with respect to it. 

Recall that \(\jump{\cdot}{\cdot}\) was defined \eqref{eq:def-jump}  to be the jump in a piecewise \(\cC^{\infty}\) function at a given point.

\begin{lemma}
    \label{lem:calc-jump-2}
    Let \(h \in \fB_0\), 
    \(k \in \bN_0\), \(0\leq j \leq N_1\).
    Then
    \begin{equation*}
        \abs{\jump{\cL_{\varphi} h}{a_{j,k+1}}} =  \abs{\varphi(a_{j,k})} \ \abs{\jump{h}{a_{j,k}}}.
    \end{equation*}
\end{lemma}

\begin{proof}
    This holds, as proved in Lemma~\ref{lem:calc-jump}, since the crucial property is that \(h\) only has discontinuities at points in \(\Delta\).
\end{proof}

\subsection{The Banach space}
We use the notation \(\apart{}\) and \(\jpart\) to denote the absolutely continuous part and jump part of the derivative respectively.
In other words, for any \(h\in \SBV\) (see, e.g.,~\cite{AFP00}), the weak derivative has the form \(\dD h = \apart{h} + \jpart{h}\) where \(\apart{h} \in L^{1}\) and \(\jpart{h}\) can be written as a sum of delta functions.
Note that when \(h\) is piecewise-\(\cC^{\infty}\) \(\apart h = h^{\prime}\), wherever the latter is well-defined.

We observe that 
\begin{equation}
    \label{eq:good-B}
    \apart \fB_0 \subseteq \fB_0,
    \quad \quad 
    \cL_{\varphi} \fB_0 \subseteq \fB_0.
\end{equation}

Recall that \(r\in \bN\) was fixed at the beginning of this section. 
Let \(\zeta = {\{\zeta_{a}\}}_{a\in\Delta}\) be a set of positive numbers (to be defined shortly).
For any \(h\in \fB_0\) we define
\begin{equation}
    \label{eq:def-custom}
    \norm{h}_{\fB_{\zeta,r}} =
    \sum_{t=0}^{r} \norm{\apart^{t} h}_{L^1}
    + \sum_{t=0}^{r-1} \norm{\apart^{t} h}_{\fJ_{\zeta}}
\end{equation}
where
\[
    \norm{h}_{\fJ_{\zeta}} = \sum_{a\in \Delta} \zeta_{a} \ \abs{\jump{h}{a}}.
\]
We refer to the first sum in this definition as the \emph{continuous} part of the norm and the second sum as the \emph{jump} part of the norm.

This norm has some resemblance to \(\BV\) or \(\SBV\).
In the case when \(r=0\), the norm corresponds to the \(L^1\) norm.
In the case when \(r=1\) and \(\zeta_a\) is equal to unity for every \(a\in \Delta\), then the norm corresponds to the usual \(\BV\) norm and the Banach space is the subset of \(\SBV\) (in turn a subset of \(\BV\)) where discontinuities are only permitted at points in \(\Delta\).
As anticipated, we now fix \(\zeta\).
For all \(1 \le j \le N_{1}\) and all \(k \in \mathbb{N}_{0}\), let
\begin{equation}
    \label{eq:def-zeta}
    \zeta(j,k) = \frac{\widetilde{\Lambda}^{k}}{|\varphi_{k}(a_{j,0})|},
\end{equation} 
where, as before, \(\varphi_{k} = \prod_{j=0}^{k-1}\varphi \circ T^{j}\). 
We hence define,
for any \(h \in \fB_{0}\),
\begin{equation}
    \label{eq:zeta-norm}
    \norm{h}_{\fJ_\zeta}
    = \sum_{j=1}^{N_1} \sum_{k=0}^{\infty} \zeta(j,k) \abs{\jump{h}{a_{j,k}}}
    + \sum_{k=1}^{N_2} \abs{\jump{h}{b_k}}.
\end{equation}
This corresponds to setting \(\zeta_{a_{j,k}} = \zeta(j,k) \), for all \(j\), \(k\) and setting \(\zeta_{a} = 1\) otherwise.

The Banach space \(\mathfrak{B}_{\zeta, r}\) is defined as the closure of \(\fB_{0}\) with respect to the \(\norm{\cdot}_{\fB_{r,\zeta}}\) norm.

\begin{lemma}
    \label{lem:well-def}
    For all \( h \in \fB_{0}\), \( \norm{h}_{\fB_{\zeta,r}} < \infty\).
\end{lemma}

\begin{proof}
    We know that~\eqref{eq:good-B}, \(\apart^{l}h \in \fB_{0}\), whenever \(0 \le l \le r\), and the restriction of \(\apart^{l}h\) to any interval of continuity 
    admits a \(\cC^{\infty}\) extension up to the boundary of such interval.
    Notice that each \(h \in \fB_{0}\) has only a finite number of discontinuities and so the jump part of the norm is finite.
    Moreover, \(\apart^{l}h\) has finite \(L^{1}\)-norm.
\end{proof}

\begin{lemma}
    \label{lem:Da-bounded}
    If \( s + l \leq r\), then
    \(\norm{\smash{\apart^{l}h}}_{\fB_{\zeta, s}} \leq \norm{h}_{\fB_{\zeta, r}}\) for all \(h\in \fB_{\zeta,r}\).
\end{lemma}

\begin{proof}
    By density it suffices to consider \(h\in \fB_0\).
    According to the definition of the norm \eqref{eq:def-custom},
    \[
        \begin{aligned}
            \norm{\smash{\apart^{l} h}}_{\mathfrak{B}_{\zeta, s}}
             & = \sum_{t=0}^{s} \norm{\smash{\apart^{l+t} h}}_{L^1}
            + \sum_{t=0}^{s-1} \norm{\smash{\apart^{l+t} h}}_{\fJ_\zeta} \\
             & \leq \sum_{t=0}^{s+l} \norm{\apart^{t} h}_{L^1}
            + \sum_{t=0}^{s+l-1} \norm{\apart^{t} h}_{\fJ_\zeta}
            = \norm{\smash{h}}_{\mathfrak{B}_{\zeta, s+l}}.
        \end{aligned}
    \]
    Since it was assumed that \( s + l \leq r\) the conclusion follows.
\end{proof}

\begin{lemma}
    \label{lem:BV}
    There exists \(C>0\) such that, for all \(l+1 \leq r\) and all \(h\in \fB_{\zeta,r}\),
    \[
        \norm{\smash{\apart^{l} h}}_{\BV}
        \leq
        C \norm{h}_{\fB_{\zeta,r}}.
    \]
\end{lemma}

\begin{proof}
    There exists \(C>1\) such that, for all \(k \in \bN\), \(\abs{\varphi_k(a_{j,0})} \leq C \widetilde{\Lambda}^{k}\) (recall that \(\widetilde{\Lambda}> \Lambda^{\sup}(T, \varphi)\), which is the maximum of \(\Lambda^{\sup}(T, \varphi, a)\) where \(a\) ranges over the set of non-trivial discontinuities) and so, using the definition of \(\zeta\), 
    \[
        \frac{1}{\zeta(j,k)}
        = \abs{\varphi_k(a_{j,0})} \widetilde{\Lambda}^{-k}
        \leq C.  
    \]
    In particular, \(1 \leq C \zeta(j,k)\).
    Consequently, for all \(h\in \fB_0\),
    \[
        \begin{split}
            \sum_{a\in \Delta}|J(h,a)|
            &\le  {C} \sum_{j=1}^{N_1} \sum_{k=0}^{\infty} \zeta(j,k)\abs{\jump{h}{a_{j,k}}}
            + \sum_{k=1}^{N_2} \abs{\jump{h}{b_k}}\\
            &\le C {\|h\|}_{\fJ_\zeta}.
        \end{split}
    \]
    Observe that (since any discontinuities of \(h\) lie in \(\Delta\)),
    \[
        \begin{aligned}
            {\|h\|}_{\BV}
             & = {\|h\|}_{L^{1}}
            + {\|\apart{h}\|}_{L^{1}}
            + \sum_{a\in \Delta} \abs{J(h,a)}, \\
            {\|h\|}_{\fB_{\zeta,1}}
             & = {\|h\|}_{L^{1}}
            + {\|\apart{h}\|}_{L^{1}}
            + {\|h\|}_{\fJ_\zeta}.
        \end{aligned}
    \]
    Hence, we have shown that, for all \(h\in \fB_0\), \({\|h\|}_{\BV} \leq C {\|h\|}_{\fB_{\zeta,1}}\).
    
    Using this together with Lemma~\ref{lem:Da-bounded} and the density of \(\fB_{0}\) implies the statement of the lemma.
\end{proof}

\subsection{Lasota-Yorke inequality}

In the next lemmas we estimate the \(\fB_{r,\zeta}\) norm of \(\cL_\varphi^n\) and hence obtain a type of Lasota-Yorke inequality.
The following is the relevant distortion result in the present setting.

\begin{lemma}%
    \label{lem:super-Da}
    There exists \(C>0\) and a set \({\{A_{l,p,n}\}}_{l,p,n}\) of functions \(A_{l,p,n}: [0,1] \to \bC\) such that,
    for all \(h\in \fB_{0}\), \(0\leq p \leq r\), \(n\in \bN_{0}\),
    \begin{equation}
        \label{eq:Dp-of-Ln}
        D_{a}^{p}\left(\cL_{\varphi}^n h\right)
        = \sum_{l=0}^{p} \cL_{\varphi}^n(A_{l,p,n} \cdot \apart^{l}h),
    \end{equation}
    the \(A_{l,p,n}\) are \(\cC^{\infty}\) on each \(\omega \in \Omega_n\),
    \(\norm{A_{l,p,n}}_{L^\infty} \leq C\)
    and \(A_{p,p,n}  = {((T^n)')}^{-p}\).
\end{lemma}
\noindent
The above lemma is proved in Section~\ref{sec:distort}.

In two common cases the derivative of \(\varphi\) is zero:
(1) In the case of the transfer operator associated to the measure of maximal entropy;
(2) In the case of the transfer operator associated to the SRB measure for a piecewise affine transformation, here \(\varphi = \frac{1}{\abs{T'}}\) and is piecewise constant.
In both these cases the terms that appear in Lemma~\ref{lem:super-Da} are:
\[
    A_{j,k,n} =
    \begin{cases}
        ((T^{n})^{\prime})^{-k} & \text{if \(j=k\)}  \\
        0         & \text{if \(j<k\)}.
    \end{cases}
\]
We don't use the above in the present work but, if required, it can be verified from the formulae given in Lemma~\ref{lem:form-of-A}.

\begin{lemma}
    \label{lem:use-of-distortion}
    There exists \(C>0\) such that, for all \(n \in \bN_{0}\), 
    \(k \ge  n\), \( 0 \le t \le r\) and \( h \in \mathfrak{B}_{0}\),
    \[
        \abs{ \jump{\apart^{t}(\cL^{n}_{\varphi}h))}{a_{j,k}} }
        \le C \sum_{l=0}^{t}
        \abs{\jump{\cL^{n}_{\varphi}(\apart^{l}h)}{a_{j,k}}}.
    \]
\end{lemma}

\begin{proof}
    By Lemma~\ref{lem:super-Da} there exist functions \(A_{l,t,n}\) (with key properties) such that
    \[
        \jump{\apart^{t}(\cL^{n}_{\varphi}h)}{a_{j,k}}
        =\sum_{l=0}^{t} \jump{\cL_{\varphi}^n(A_{l,t,n} \cdot \apart^{l}h)}{a_{j,k}},
    \]
    and, by definition of \(\cL_\varphi\),
    \[
        \jump{\cL_{\varphi}^n(A_{l,t,n} \cdot \apart^{l}h)}{a_{j,k}}
        = \sum_{\omega\in\Omega_n}\jump{(A_{l,t,n} \cdot \varphi_n\cdot \apart^l h) \circ T^{n}|_{\omega}^{-1}}{a_{j,k}}.
    \]
    Lemma~\ref{lem:delta} implies that for all \(a_{j,k}\), \(k \ge n\), there exists a unique \(\omega^\prime\in\Omega_n\) such that \({T^n|_{\omega^\prime}}^{-1}(a_{j,k})= a_{j,k-n}\in\Delta\)
    and for \(\omega\neq \omega^\prime\), \({T^n|_{\omega}}^{-1}(a_{j,k})\notin\Delta\).  
    Furthermore, 
    \begin{multline*}
        \sum_{\omega\in\Omega_n}\jump{{(A_{l,t,n} \cdot \varphi_n\cdot \apart^l h) \circ {T^{n}|_{\omega}}^{-1}}}{a_{j,k}}\\
        = A_{l,t,n}(a_{j,k-n}) \jump{(\varphi_n\cdot \apart^l h)\circ {T^n|_{\omega^\prime}}^{-1}}{a_{j,k}}.
    \end{multline*}
    This is because \(A_{l,t,n} \cdot \varphi_{n}\) is continuous at any point in
    \(T^{-n}(a_{j,k})\), \(k \ge n\) (Lemma~\ref{lem:disc-of-B} \& Lemma~\ref{lem:super-Da}), and, \(\apart^{l}h\) is discontinuous only in \(\Delta\) \eqref{eq:good-B}.
    Similarly,
    \[
        \begin{aligned}
            \jump{(\varphi_n\cdot \apart^l h)\circ {T^n|_{\omega^\prime}}^{-1}}{a_{j,k}}
             & = \sum_{\omega\in\Omega_n}\jump{{(\varphi_n\cdot \apart^l h) \circ {T^{n}|_{\omega}}^{-1}}}{a_{j,k}} \\
             & = \jump{\cL^{n}_{\varphi}(\apart^{l}h)}{a_{j,k}}.
        \end{aligned}
    \]
    Collecting together the above equalities we have shown that
    \[
        \abs{\jump{\cL_{\varphi}^n(A_{l,t,n} \cdot \apart^{l}h)}{a_{j,k}}}
        \leq 
        \norm{A_{l,t,n}}_{L^\infty}
        \abs{\jump{\cL^{n}_{\varphi}(\apart^{l}h)}{a_{j,k}}}.
    \]
    The statement of the lemma follows since Lemma~\ref{lem:super-Da} tells that \(\norm{A_{l,t,n}}_{L^\infty}\) is bounded uniformly in \(n\).
\end{proof}

For convenience we use the following notation for the continuous part and jump part of the norm, respectively,
\[
    \norm{h}_{\fC^r}
    = \sum_{j=0}^{r}{\|\apart^{j} h\|}_{L^{1}},
    \quad \quad
    \norm{h}_{\fJ_\zeta^r}
    = \sum_{j=0}^{r-1}{\|\apart^{j} h\|}_{\fJ_\zeta}.
\]
By change of variables,
since \(\norm{\varphi \cdot \sT^\prime}_{L^\infty}\) is finite,
for all \(h\in L^1\), \(n\in \bN_0\),
\begin{equation}
    \label{eq:L^{1}-contraction}
    \norm{\cL_{\varphi}h}_{L^1} \leq \norm{\varphi_n \cdot \smash{\left(\sT^n\right)^\prime} \cdot h}_{L^1}.
\end{equation}
\begin{lemma}
    \label{lem:boundedness}
    There exists some \(C > 0\) such that, for all \(h \in \fB_0\),
    \[
        {\|\cL_{\varphi} h\|}_{\fB_{\zeta,r}} \le C {\|h\|}_{\fB_{\zeta,r}}.
    \]
    In particular, \(\cL_{\varphi}\) extends to a continuous operator on \(\mathfrak{B}_{\zeta,r}\).
\end{lemma}

\begin{proof}
    We estimate
    \({\|\cL_{\varphi}h\|}_{\fB_{\zeta,r}}
    = \sum_{l=0}^{r}{\|\apart^{l}\cL_{\varphi}h\|}_{L^{1}}
    + \sum_{l=0}^{r-1}{\|\apart^{l}\cL_{\varphi}h\|}_{\fJ_\zeta}\).
    For any \(0 \le l \le r\), using Lemma~\ref{lem:super-Da} (for \(n=1\)) and \eqref{eq:L^{1}-contraction},
    \[
        \begin{aligned}
            {\|\apart^{l}\cL_{\varphi}h\|}_{L^{1}}
             & \le \sum_{s=0}^{l} {\|\cL_{\varphi}(A_{s,l,1}\apart^{s}h)\|}_{L^{1}}        \\
             & \le  \sum_{s=0}^{l} {\| \varphi \cdot \sT' \cdot  A_{s,l,1} \cdot \apart^{s}h\|}_{L^{1}}
            \le C \sum_{s=0}^{l} {\|\apart^{s}h\|}_{L^{1}},
        \end{aligned}
    \]
    for some \(C>0\).
    This takes care of the continuous part of the norm.
    As for the jump part, we have, for any \(0 \le l \le r-1\)
    \[
        \begin{aligned} 
            {\|\apart^{l}\cL_{\varphi}h\|}_{\fJ_\zeta}
            &= \sum_{j=1}^{N_1} \sum_{k=1}^{\infty} \zeta(j,k) \abs{\jump{\apart^{l}(\cL_{\varphi}h)}{a_{j,k}}}\\
            & \ \ + \left( \sum_{k=1}^{N_2} \abs{\jump{\apart^{l}\cL_{\varphi}h}{b_k}}
            + \sum_{j=1}^{N_1} \abs{\jump{\apart^{l}(\cL_{\varphi}h)}{a_{j,0}}} \right).
        \end{aligned}
    \]
    Let us start with the final pair of sums.
    Crucially this is a finite number of terms.
    By Lemma~\ref{lem:BV} and Lemma~\ref{lem:super-Da}, we have that, for some \(C>0\),
    \begin{equation}
        \label{eq:finite-sum}
        \begin{aligned}
            \abs{\jump{\apart^{l}\cL_{\varphi}h}{b_k}}
             & \le \sum_{s=0}^{l}
             \abs{\jump{\cL_{\varphi}(A_{s,l,1} \cdot \apart^{s}h)}{b_k}}  \\
             & \le 2  \ \# \Omega \sum_{s=0}^{l} {\|\varphi \cdot  A_{s,l,1} \cdot \apart^{s}h\|}_{L^{\infty}}
            \le C {\|h\|}_{\fB_{\zeta,r}}.
        \end{aligned}
    \end{equation}
    An identical estimate holds for the terms involving \(\abs{\jump{\apart^{l}\cL_{\varphi}h}{a_{j,0}}}\).
    This takes care of the finite sum.

    For the first sum, the remaining term to be estimated, Lemma~\ref{lem:use-of-distortion} implies, 
    \[
        \abs{\jump{\apart^{l}(\cL_{\varphi}h)}{a_{j,k}}} \le C \sum_{s=0}^{l}  \abs{\jump{\cL_\varphi(\apart^{s}h)}{a_{j,k}}}.
    \]
    Lemma~\ref{lem:calc-jump-2} implies that this is bounded by
    \( C \sum_{s=0}^{l}  \abs{\jump{\apart^{s}h}{a_{j,k-1}}}\).
    Summing over \(k\), this means that
    \[
        \sum_{j=1}^{N_1} \sum_{k=0}^{\infty} \zeta(j,k) \abs{\jump{\apart^{l}(\cL_{\varphi}h)}{a_{j,k}}}
        \leq C \norm{h}_{\fJ_\zeta^r},
    \]
    as required.
\end{proof}

Subsequently we will prove inequalities concerning \(\cL_{\varphi}\) and Lemma~\ref{lem:boundedness} means that it is sufficient to prove them on \(\fB_{0}\).

\begin{lemma}
    \label{lem:est-c-part}
    For each \(n\in  \bN\), there exists \(C_n >0\) such that,
    if \(h \in \fB_{r,\zeta}\),
    \[
        \norm{\cL_{\varphi}^n h}_{\fC^r}
        \le C_n \norm{h}_{\fC^{r-1}} +
        \norm{\varphi_n}_{L^\infty}
        \norm{{1}/{(T^n)'}}_{L^\infty}^{r-1}
        \norm{h}_{\fC^r}.
    \]
\end{lemma}

\begin{proof}
    By density it suffices to prove the statement for \(h\in \fB_0\).
    Using Lemma~\ref{lem:super-Da}, and \eqref{eq:L^{1}-contraction}, one has that, for \(0 \le  q \le r-1\),
    \begin{equation*}
        \begin{aligned}
            \norm{\apart^{q}{(\cL^{n}_{\varphi} h)}}_{L^{1}}
             & \le
            \sum_{l = 0}^{q}\norm{\cL^{n}_{\varphi} \cdot (A_{l,q,n} \cdot \apart^{l}{h})}_{L^{1}}
            \le
            \sum_{l = 0}^{q}\norm{ \varphi_{n} (T^{n})^{\prime} \cdot A_{l,q,n}
            \cdot \apart^{l}{h}}_{L^{1}}                                                    \\
             & \le \sum_{l= 0}^{q}  \norm{\varphi_{n}\cdot (T^{n})^{\prime} \cdot A_{l,q,n}}_{L^{\infty}}
            \norm{\smash{\apart^{l}{h}}}_{L^{1}}.
        \end{aligned}
    \end{equation*}
    Lemma~\ref{lem:super-Da} and the assumption \(\norm{\varphi \cdot T^{\prime}}_{L^{\infty}}< \infty\) implies that 
    \[
        C_n 
        = \max_{l\le q \le r}\norm{\varphi_n\cdot (T^{n})^{\prime} \cdot A_{l,q,n}}_{L^{\infty}}
    \] 
    is finite.
    This means that, 
    \(\norm{\apart^{q}{(\cL^{n}_{\varphi} h})}_{L^{1}} \le C_n \norm{h}_{\fC^q}\)
    and, in particular,
    \[
        \norm{\cL_{\varphi} h}_{\fC^{r-1}} \le r C_n \norm{h}_{\fC^{r-1}}.
    \]
    As for the top term, using
    \(A_{r,r,n} =  {((\sT^{n})^{\prime})^{-r}}\)
    (Lemma~\ref{lem:super-Da}),  
    \begin{equation*}
        \begin{split}
            \norm{\apart^{r}{(\cL_{\varphi} h)}}_{L^{1}}
            &\le
            \norm{\varphi_n\cdot (T^{n})^{\prime} \cdot A_{r,r,n}}_{L^{\infty}}
            \norm{\smash{\apart^{r}{h}}}_{L^{1}}\\
            & \ \ +
            \sum_{l= 0}^{r-1}  \norm{\varphi_n\cdot (T^{n})^{\prime} \cdot A_{l,r,n}}_{L^{\infty}}
            \norm{\smash{\apart^{l}{h}}}_{L^{1}}\\
            &\le
            \norm{\smash{{\varphi_n}\cdot {((T^{n})^{\prime})^{-(r-1)}}}}_{L^{\infty}}
            \norm{h}_{\fC^r}
            +
            r C_n \norm{h}_{\fC^{r-1}}.
        \end{split}
    \end{equation*}
    The statement follows by summing the above estimates.
\end{proof}

\begin{lemma}
    \label{lem:continuity}
    Let \(n\in \bN\).
    There exists \(C_n > 0\) such that,
    for all \(h \in \fB_0\), \(x\in [0,1]\) and \(0\le t \le r-1\),
    \[
        |J(\apart^{t}(\cL^{n}_{\varphi}h),x)|
        \le C_n {\|h\|}_{\fB_{\zeta,r}}.
    \]
\end{lemma}

\begin{proof}
    By Lemma~\ref{lem:BV} and Lemma~\ref{lem:boundedness}, 
    for all \(h \in \fB_0\) and \(x\in [0,1]\),
    \[
        |J(\apart^{t}(\cL^{n}_{\varphi}h),x)| 
        \le 2 {\|\apart^{t}(\cL^{n}_{\varphi}h)\|}_{L^{\infty}} 
        \le 2 C {\|\cL^{n}_{\varphi}h\|}_{\fB_{\zeta,r}} \le 2 {C} C^n {\|h\|}_{\fB_{\zeta,r}}. \qedhere
    \]
\end{proof}

\begin{lemma}
    \label{lem:est-j-part}
    There exists \(C>0\) and, for each \(n \in \mathbb{N}\), there exists a finite set of linear functionals \(\{\rho_{l}\}_{l\in Q_{n}}\) (on \(\fB_{r,\eta}\)) such that, for all \(h \in \fB_{r,\eta}\),
    \[
        \norm{\cL_{\varphi}^{n} h}_{\fJ_\zeta^r}
        \le
        \sum_{l \in Q_{n}}|\rho_{l}(h)|
        + C \tilde\Lambda^{n} \norm{h}_{\fJ_\zeta^r}.
    \]
\end{lemma}

\begin{proof}
    By density it suffices to consider \(h \in \fB_0\).
    For all \(0\leq t \leq r-1\), \(n\in \bN_0\), using only the definition of \(\norm{\cdot}_{\fJ_\zeta}\),  
    \begin{equation}
        \label{eq:zeta-norm-higher}
        \begin{split}
            {\|\apart^{t}(\cL^{n}_ {\varphi}h)\|}_{\fJ_\zeta}
            &= \sum_{k=1}^{N_{2}}|J(\apart^{t}(\cL^{n}_{\varphi}h),b_{k})| \\
            & \ \ + \sum_{j=1}^{N_{1}}\sum_{k=0}^{n-1}\zeta(j,k)|J(\apart^{t}(\cL^{n}_{\varphi}h),a_{j,k})| \\
            & \ \ + \sum_{j=1}^{N_{1}}\sum_{k=n}^{\infty}\zeta(j,k) \abs{J(\apart^{t}(\cL_{\varphi}^n h)),a_{j,k})}.
        \end{split}
    \end{equation}
    For each \(n,t,j,k\), let
    \[
        \begin{aligned}
            \widetilde{\rho}_{n,t,k}(h) 
            &= J(\apart^{t}(\cL^{n}_{\varphi}h),b_{k}), \\
            \rho_{n,t,j,k}(h) 
            &= \zeta(j,k)(J(\apart^{t}(\cL^{n}_{\varphi}h),a_{j,k})).
        \end{aligned}
    \]
    Lemma~\ref{lem:continuity} implies that these are all linear functionals on \(\fB_{\zeta,r}\).
    This accounts for the term \(\sum_{l \in Q_{n}}|\rho_{l}(h)|\) in the statement of the lemma.
    It remains to estimate the final term of \eqref{eq:zeta-norm-higher}.
    We notice that, applying Lemma~\ref{lem:calc-jump-2} to \(\apart^{t}h\), for \(k\geq 1\), \(0\leq j \leq N_1\),
    \[
        \abs{\jump{\cL_{\varphi} (\apart^{t}h)}{a_{j,k}}} =  \abs{\varphi(a_{j,k-1})} \ \abs{\jump{\apart^{t}h}{a_{j,k-1}}}.
    \]
    Iterating this we obtain that, whenever \(k \geq n\),
    \begin{equation}
        \abs{\jump{\cL^n_{\varphi} (\apart^{t}h)}{a_{j,k}}} =  \abs{\varphi_n(a_{j,k-n})} \ \abs{\jump{\apart^{t}h}{a_{j,k-n}}}.
    \end{equation}
    Using also Lemma~\ref{lem:use-of-distortion},
    \[
        \begin{aligned}
            \abs{\jump{\apart^{t}(\cL_{\varphi}^{n}h)}{a_{j,k}}}
             & \le C \sum_{i=0}^{t} \abs{\jump{\cL_{\varphi}^{n}(\apart^{i}h)}{a_{j,k}}} \\
             & = C \abs{\varphi_n(a_{j,k-n})} \sum_{i=0}^{t}   \abs{\jump{\apart^{i}h}{a_{j,k-n}}}.  
        \end{aligned}
    \]    
    It follows by definition \eqref{eq:def-zeta} that
    \( \zeta(j,k) \abs{\varphi_n(a_{j,k-n})} = \tilde{\Lambda}^n \zeta(j,k-n)\).
    Consequently,
    \[
            \zeta(j,k) \abs{\jump{\apart^{t}(\cL_{\varphi}^{n}h)}{a_{j,k}}}
            \le C \tilde{\Lambda}^n \zeta(j,k-n) \sum_{i=0}^{t}   \abs{\jump{\apart^{i}h}{a_{j,k-n}}}.  
    \] 
    Summing over \(j\) and \(k\geq n\),
    we have shown that the final term of \eqref{eq:zeta-norm-higher} is bounded above by
    \[
        C \tilde{\Lambda}^n \sum_{j=1}^{N_{1}}\sum_{k=0}^{\infty} 
        \zeta(j,k) \sum_{i=0}^{t}   \abs{\jump{\apart^{i}h}{a_{j,k}}}
        = C \tilde\Lambda^{n} \norm{h}_{\fJ_\zeta^t}.
    \]  
    The proof is completed by summing all the estimates over \(t\).
\end{proof}

We are now in the position to combine the estimates obtained above.
Since we fixed \(r\) such that \(\eta \lambda^{r-1} < \widetilde{\Lambda}\), we can choose \(\widetilde{\eta} > \eta\) and \(\widetilde{\lambda} > \lambda\) such that  \(\widetilde{\eta} \widetilde{\lambda}^{r-1} \leq \widetilde{\Lambda}\).
Let such \(\widetilde{\eta}\) and \(\widetilde{\lambda}\) be fixed for the remainder of the section.

\begin{lemma}
    \label{lem:LY}
    There exists \(C>0\) and, for all \(n\in \bN\), there exists \(C_n>0\)
    and a finite set of linear functionals \({\{\rho_{j}\}}_{j\in Q_n}\), such that, for all \(h \in \fB_0\),
    \[
        \begin{split}
            \norm{\cL_{\varphi}^n h}_{\fB_{\zeta,r}}
            &\le
            C_n \norm{h}_{\fC^{r-1}} +
            C \tilde{\eta}^{n}
            \tilde{\lambda}^{n(r-1)}
            \norm{h}_{\fC^r} \\
            & \ \ + \sum_{j \in Q_{n}}|\rho_{j}(h)|
            + C \widetilde{\Lambda}^n \norm{h}_{\fJ_\eta^r}.
        \end{split}
    \]
\end{lemma}

\begin{proof}
    Since \(\tilde{\eta} > \eta\) and \(\tilde{\lambda} > \lambda\), Lemma~\ref{lem:est-c-part} implies that, for some \(C>0\),
    \[
        \norm{\cL_{\varphi}^n h}_{\fC^r}
        \le C_n \norm{h}_{\fC^{r-1}} +
        C \tilde{\eta}^{n}
        \tilde{\lambda}^{n(r-1)}
        \norm{h}_{\fC^r},
    \]
    as required for the first and second term in the statement of the lemma.
    The third and fourth terms are given by Lemma~\ref{lem:est-j-part}.
\end{proof}

\subsection{Essential spectral radius}
Here we use the estimate of Lemma~\ref{lem:LY}, together with a compactness result, in order to estimate the essential spectral radius.

The following is a type of compact embedding between \(\norm{\cdot}_{\fB_{\zeta,r}}\) and \(\norm{\cdot}_{\fC^s}\).

\begin{lemma}
    \label{lem:compact}
    Suppose that \(s\in \bN_0\), \(p\in\bN\), \(s < p\).
    For every \(\epsilon>0\) there exists a finite set \({\{S_k\}}_{k}\) of subsets of \(\fB_{\zeta,p}\) such that \(\{h: \norm{h}_{\fB_{\zeta,p}} \leq 1\}\) is covered by \(\bigcup_{k} S_{k}\) and,
    for each \(k\), \(\norm{h-g}_{\fC^s}\leq \epsilon\) for every \(h,g \in S_k\).
\end{lemma}

\begin{proof}
    By Lemma~\ref{lem:Da-bounded} and Lemma~\ref{lem:BV} which implies that \(\{h : \norm{\smash{\apart^l} h}_{\fB_{\zeta,1}}  \leq 1 \}\) is a subset of  \(\{h : \norm{\smash{\apart^l} h}_{\BV}  \leq C \}\),
    \[
        \left\{h : \norm{h}_{\fB_{\zeta,p} \leq 1} \right\}
        \subseteq
        \bigcap_{l = 0}^{p-1} 
        \left\{h : \norm{\smash{\apart^l} h}_{\BV}  \leq C \right\}.
    \]
    Fix, for the moment, \(l\) and consider the set \(\left\{h : \norm{\smash{\apart^l} h}_{\BV}  \leq C \right\}\).
    The compactness of \(\BV\) in \(L^1\) implies that, for all \(\epsilon>0\) there exists a finite cover \({\left\{ S_k \right\}}_k\) of \(\left\{h : \norm{h}_{\BV}  \leq C \right\}\) such that, \(\norm{f - g}_{L^1} \leq \epsilon\) whenever \(f,g \in S_k\) for each \(k\).
    The sets \(R_k = \{h : \apart^l h \in S_k \}\) 
    are a finite cover of \(\{h : \norm{\smash{\apart^l} h}_{\BV}  \leq C \}\) such that \(\norm{\apart^l(f - g)}_{L^1} \leq \epsilon\) whenever \(f,g \in R_k\) for each \(k\).
    This can be done for all \(0 \leq l \leq p-1\) and so would imply that \(\norm{f - g}_{\fC^{p-1}} \leq \epsilon\).
    This implies that the statement holds true for all \(s<p\).
\end{proof}
\begin{lemma}
    \label{lem:ess-spec}
    The essential spectral radius of \(\cL_{\varphi}\) on \(\fB_{\zeta,r}\) is at most \(\widetilde{\Lambda}\).
\end{lemma}

\begin{proof}
    Let \(B_0 = \left\{ \smash{h \in \fB_{\zeta,r} : \norm{h}_{\fB_{\zeta,r}} \leq 1} \right\}\)
    and, for \(n\in\bN\), let \(B_n = \cL_{\varphi}^n B_0\).
    We will show that, for each \(n\in \bN\) large, \(B_n\) admits a finite cover such that the diameter of each element of the cover, measured in the \(\norm{\cdot}_{\fB_{\zeta,r}}\)-norm, is not greater than \(C\widetilde{\Lambda}^n\).
    According to Nussbaum~\cite{Nussbaum70}, this suffices to prove that the essential spectral radius of \(\cL_{\varphi}:\fB_{\zeta,r} \to \fB_{\zeta,r}\) is not greater than \(\widetilde{\Lambda}\).

    Since \(\widetilde{\eta}\widetilde{\lambda}^{r-1} \leq \widetilde{\Lambda}\),
    Lemma~\ref{lem:LY} tells that, for all \(h \in \fB_{\zeta,r}\), \(n\in\bN\),
    \[
        \begin{aligned}
            \norm{\cL_{\varphi}^n h}_{\fB_{\zeta,r}}
            \le
            C \widetilde{\Lambda}^n \norm{h}_{\fB_{\zeta,r}}
            + C_n \norm{h}_{\fC^{r-1}}
            + \sum_{j\in Q_n} |\rho_{j}(h)|.
        \end{aligned}
    \]
    Fix \(n\in \bN\).
    Using the compact embedding of Lemma~\ref{lem:compact} we can choose a finite set \({\{S_k\}}_k\) of subsets of \(B_0\) such that
    \[
        C_n \norm{h-g}_{\fC^{r-1}} \leq \widetilde{\Lambda}^n,
        \quad \quad
        \text{for all  \(h,g \in S_k\), for each \(k\)}.
    \]
    Observe that the set \( \{ \sum_{j\in Q_{n}} \rho_{j}(h) :h \in B_0 \} \subset \bC\) is bounded because \(B_0\) is bounded and the \(\rho_{j}\) are  linear functionals.
    This means that we can further refine the finite covering \({\{S_k\}}_k\) of \(B_0\) such that it also satisfies
    \[
        \sum_{j \in Q_{n}}|\rho_{j}(h-g)|\leq \widetilde{\Lambda}^n,
        \quad \quad
        \text{for all  \(h,g \in S_k\), for each \(k\)}.
    \]

    Observe that \({\{\cL^n_\varphi S_k\}}_k\) is a finite cover of \(B_n\).
    For each \(k\), the combination of the above estimates implies that, for all \(h,g\in S_k\),
    \[
        \norm{\cL^n_\varphi h - \cL^n_\varphi g}_{\fB_{\zeta,r}}
        = \norm{\cL^n_\varphi (h - g)}_{\fB_{\zeta,r}}
        \leq (C+2) \widetilde{\Lambda}^n.
    \]
    I.e., as required, we have estimated the diameter of each of the sets \(\cL^n_\varphi S_k\).
\end{proof}

This concludes the proof of Theorem~\ref{thm:custom} since the three statements of the theorem are given by Lemma~\ref{lem:well-def}, Lemma~\ref{lem:boundedness} and Lemma~\ref{lem:ess-spec}, respectively. 
The inclusion in \(L^{\infty}\) is given by Lemma~\ref{lem:BV}.

%%%%%%%%%%%%%%%%%%%%%%%%%%%%%%%%
\section{Distortion}
\label{sec:distort}
%%%%%%%%%%%%%%%%%%%%%%%%%%%%%%%%

This section is devoted to the proof of Lemma~\ref{lem:super-Da}. 
We work in the same setting as Section~\(\ref{sec:custom}\), but none of this section relies on previous results.

\begin{lemma}%
    \label{lem:deriv-L}
    For all \(h\in \fB_0\), 
    \[
        \apart{\left(\cL_{\varphi}h\right)}
        = \cL_{\varphi}\left(\tfrac{\varphi^{\prime}}{\varphi \cdot \sT'} \cdot h\right)
        + \cL_{\varphi}\left(\tfrac{1}{\sT'} \cdot \apart{h} \right).
    \]
\end{lemma}

\begin{proof}
    Since
    \( \cL_{\varphi}h  =  \sum_{k=1}^{n} \left(\varphi \cdot h\right) \circ  {\left.\sT\right|}_{\omega_{k}}^{-1} \cdot \characteristic{\sT\omega_k}\),
    we calculate the derivative of each summand, restricted to \(\sT\omega_k\),
    \[
        \begin{aligned}
            \apart \left( \left(\varphi \cdot h\right) \circ {\left.\sT\right|}_{\omega_{k}}^{-1}\right)
             & =
            \frac{\apart{\left(\varphi \cdot h\right)}}{\sT'} \circ {\left.\sT\right|}_{\omega_{k}}^{-1} \\
             & =
            \left(\frac{ \varphi^{\prime}}{\sT'} \cdot h\right) \circ {\left.\sT\right|}_{\omega_{k}}^{-1}
            +
            \left(\frac{\varphi}{\sT'} \cdot \apart{h} \right) \circ {\left.\sT\right|}_{\omega_{k}}^{-1}.
        \end{aligned}
    \]
    Summing over \(k\) we obtain the result.
\end{proof}

\begin{lemma}
    \label{lem:form-of-A}
    There exist functions \(A_{l,p,n}: [0,1] \to \bC\), \(0 \le p \le r\) and \(0 \le l \le p\) such that 
    \[
        D_{a}^{p}\left(\cL_{\varphi}^n h\right)
    = \sum_{l=0}^{p} \cL_{\varphi}^n(A_{l,p,n} \cdot \apart^{l}h).
    \]
    Furthermore, these functions are \(\cC^{\infty}\) on \(\omega\) for each \(\omega \in \Omega_n\) and satisfy,
    \begin{equation}
        \label{eq:A_lpn}
        \begin{aligned}
            A_{0,p+1,n}
             & = \frac{A_{0,p,n}^{\prime}}{(T^n)'}
            + \frac{A_{0,p,n}}{(T^{n})^{\prime}}
            \cdot  \frac{\varphi_{n}^{\prime}}{\varphi_{n}}, \\
            A_{l,p+1,n}
             & = \frac{A_{l,p,n}^{\prime}}{(T^n)'}
            +\frac{A_{l,p,n}}{(T^{n})^{\prime}}
            \cdot \frac{\varphi_{n}^{\prime}}{\varphi_{n}}
            + \frac{A_{l-1,p,n}}{(T^n)'},
            \quad \text{when \(1\leq l \leq p\)},             \\
            A_{p+1,p+1,n}
             & =  \frac{A_{p,p,n}}{(T^n)'}.
        \end{aligned}
    \end{equation}
\end{lemma}

\begin{proof}
    Let \(n\in \bN\).
    First observe that \(A_{0,0,n} = 1\) and, in particular, is \(\cC^{\infty}\) on \(\omega\) for each \(\omega\).
    In other words the lemma holds for the case \(p=0\).
    Now suppose that the lemma holds for some fixed \(p\).
    This means that,    
    \[
        \begin{aligned}
            \apart^{p+1}\left(\cL^{n}_{\varphi} h\right)
             & =  \sum_{\ell=0}^{p} \apart \cL^{n}_{\varphi}(A_{\ell,p,n} \cdot \apart^{\ell} h) \\
             & = \sum_{\ell =0}^{p}
            \sum_{\omega \in \Omega_{n}}
            \apart  \left(   (\varphi_{n} \cdot A_{\ell,p,n} \cdot \apart^{l}h ) \circ {\left.\sT^{n}\right|}_{\omega}^{-1} \right)\cdot \characteristic{\sT^{n}\omega}.
        \end{aligned}
    \]
    By Lemma~\ref{lem:deriv-L} the term \(\apart  \left(   (\varphi_{n} \cdot A_{\ell,p,n} \cdot \apart^{l}h ) \circ {\left.\sT^{n}\right|}_{\omega}^{-1} \right)\) is equal to
    \[
        \left(
        \frac{\varphi_{n}}{(T^{n})'}
        \left[A_{\ell,p,n}^{'} \cdot \apart^{\ell}h
                +  A_{\ell,p,n} \cdot \apart^{\ell +1}h
                + \frac{\varphi_{n}'}{\varphi_{n}} \cdot A_{\ell,p,n}
                \cdot \apart^{\ell}h  \right]
        \right) \circ {\left.\sT^{n}\right|}_{\omega}^{-1}.
    \]
    Comparing this with the expression for \(D_{a}^{p+1}\cL_{\varphi}^n(h)\) (i.e., collecting same order derivatives together), we obtain the claimed relationship between the terms \eqref{eq:A_lpn}.
    That the \(A_{l,p,n}\) are \(\cC^{\infty}\) in each \(\omega \in \Omega_n\) follows from the proven relationships \eqref{eq:A_lpn}.
    Indeed, both \(\varphi_n\) and \((T^{n})^{\prime}\) are so and bounded away from zero by assumption.
\end{proof}

The following is \cite[Lemma 6.1]{GL06} transposed to the present setting.

\begin{lemma}
    \label{lem:distort-crux}
    Let \({\{I_k\}}_{k=0}^{\infty}\) be intervals and let \(S_k : I_k \to I_{k+1}\) be \(\cC^{r}\) contracting transformations with uniform contraction\footnote{I.e., there exists \(\lambda < 1\), \(C>0\) such that \((S^{k})^{\prime}(x) \le C\lambda^{k}\) for all \(k \in \mathbb{N}\), \(x \in I_{0}\).}.
    Furthermore let \(\theta_{k}:I_k\to (0,\infty)\) be \(\cC^{r}\) functions.
    Suppose that there exists \(C>0\) such that, for any \(x \in I_k\) for any \(k\), the \(\cC^{r}\) norm of \(\theta_{k}\) is bounded on a neighbourhood of \(x\) by \(C\theta_{k}(x)\). 
    Let \(S^{k} = S_{k-1} \circ \cdots \circ S_{0}\).
    
    There exists \(L>0\) such that, for all \(x \in I_0\),
    \begin{equation*}
        \norm{ \textstyle \prod_{k=0}^{n-1}\theta_{k}\circ S^{k} }_{\cC^{r}(I_0)} 
        \le 
        L \textstyle  \prod_{k=0}^{n-1}\theta_{k}\circ S^{k}(x).
    \end{equation*}
\end{lemma}

\begin{lemma}
    \label{lem:diff-quotient}
    There exists \(C > 0\) such that the following holds.
    If \(A\) is a positive \(\cC^r\) function defined on an interval, such that for all \(x\), the \(\cC^r\) norm of \(A\) in a neighbourhood of \(x\) is bounded by \(L A(x)\), for some \(L\), then, 
    \[
        \norm{{A^\prime}/{A}}_{\cC^{r-1}}
        \leq CL^{r}.
    \]
\end{lemma}

\begin{proof}
    The \(q\)\textsuperscript{th} derivative of the quotient \(\frac{A^\prime}{A}\) is equal to a sum of terms of the form
    \[
        \frac{A^{(q_1)}}{A} \cdot \frac{A^{(q_2)}}{A} \cdots \frac{A^{(q_j)}}{A}
    \]  
    where the \(q_1 + q_j + \cdots + q_j = q+1\).
    That this is true for \(q=0\) is immediate. 
    If we assume that the statement is true for \(q\), then we can observe that differentiating only produces terms of the same type. 
    Consequently \(\norm{{A^\prime}/{A}}_{\cC^{r-1}}\) is bounded by \(L^r\) multiplied by a combinatorial constant \(C\) which depends only on \(r\).
\end{proof}

Lemma~\ref{lem:form-of-A} proved the existence of the \(A_{l,p,n}\) and, from the formula~\eqref{eq:A_lpn} for \(A_{p+1,p+1,n}\), the fact that \(A_{p,p,n}  = {((T^n)')}^{-p}\).
To complete the proof of Lemma~\ref{lem:super-Da}, it remains to show that there exists \(C>0\) such that, for \(\norm{A_{l,p,n}}_{L^\infty}\le C\), in particular that the bound is independent of \(n\).

Let \(n\in \bN_0\), \(\omega \in \Omega_n\), \(x\in\omega\).
Since \({\left. T^n \right|}_{\omega}\) is invertible,
\begin{equation}
    \label{eq:def-B}
    B_{l,p,n,\omega} = A_{l,p,n}\circ {\left. T^n \right|}_{\omega}^{-1}
\end{equation}
is well defined on \(\sT^n \omega\).

\begin{lemma}
    \label{lem:control-B}
    Let \(0\leq p \leq r\).
    There exists \(C_p > 0 \) such that,
    for all \(n \in \bN_0\), \(\omega \in \Omega_n\), \(0\leq l \leq p\),
    \[
        \norm{B_{l,p,n,\omega}}_{\cC^{r-p}(\sT^n \omega)} \leq C_p.
    \]
\end{lemma}

\begin{proof}
    First note that \(B_{0,0,n} = 1\) for all \(n\in \bN_0\) and so the statement of the lemma holds for \(p=0\).

    Proceeding by induction on \(p\), let us suppose that the statement of the lemma is true for \(p\).
    By Lemma~\ref{lem:form-of-A},
    \[
        A_{l,p+1,n}
        = \frac{A_{l,p,n}^{\prime}}{(T^n)'}
        + \frac{A_{l,p,n}}{(T^{n})^{\prime}}
        \cdot \frac{\varphi_{n}^{\prime}}{\varphi_{n}}
        + \frac{A_{l-1,p,n}}{(T^n)'},
    \]
    and consequently,
    \begin{equation}
        \label{eq:iterate-B}
        B_{l,p+1,n,\omega}
        = B_{l,p,n,\omega}^{\prime}
        + B_{l,p,n,\omega} \cdot \left ( {\Phi_n^\prime} / {\Phi_n} \right )
        + B_{l-1,p,n,\omega} \cdot  \Theta_n,
    \end{equation}
    where, for convenience we write \(\Phi_n = \varphi_n \circ {\left. T^n \right|}_{\omega}^{-1}\) and \(\Theta_n = \frac{1}{(T^n)'}\circ {\left. T^n \right|}_{\omega}^{-1}\).
    (In the case \(l=0\) or \(l=p+1\) not all of the three terms are present but this only makes the following estimate easier.)
    Taking the \(\cC^{r-(p+1)}\) norm%
    \footnote{\(\cC^{r}\) is the closure of \(\cC^{\infty}\) with respect to the norm \({\|h\|}_{\cC^{r}} = \sup_{0 \le k \le r}2^{r-k} {\|h^{(k)}\|}_{L^{\infty}}\)and so \((\cC^{r}, {\|\cdot\|}_{\cC^{r}})\) is a Banach algebra.} %
    and using the inductive hypothesis,
    \[
        \norm{B_{l,p+1,n}}_{\cC^{r-(p+1)}}
        \leq
        C_p \left(2 + \norm{{\Phi_n^\prime} / {\Phi_n}}_{\cC^{r}}   +    \norm{ \Theta_n}_{\cC^{r}}  \right).
    \]
    It was assumed that \(\varphi\) is \(\cC^{\infty}(\mathcal{I})\) and bounded away from zero, consequently there exists \(C>0\) such that, for all \(x\), the \(\cC^r\) norm of \(\varphi\) in a neighbourhood of \(x\) is bounded by \(C \varphi(x)\).
    The same property holds also for \(\frac{1}{\sT^\prime}\).
    This means that Lemma~\ref{lem:distort-crux} applies in both cases. 
    This yields directly \(  \norm{ \Theta_n}_{\cC^{r}}  \leq L \). 
    As for the quotient \({\Phi_n^\prime} / {\Phi_n}\), we also apply Lemma~\ref{lem:diff-quotient} for the family of functions \(\Phi_{n}\), which yields \(\norm{{\Phi_n^\prime} / {\Phi_n}}_{\cC^{r}}  \leq CL^{r} \), independently on \(n\).
\end{proof}

Let \(n\in \bN_0\), \(\omega \in \Omega_n\), \(x \in \omega\).
By definition \eqref{eq:def-B},
\( A_{l,p,n}(x) = B_{l,p,n,\omega}(\sT^n x) \).
Lemma~\ref{lem:control-B} provides, in particular, a uniform \(L^{\infty}\) bound on \(B_{l,p,n,\omega}\) and this concludes the proof of Lemma~\ref{lem:super-Da}.

%%%%%%%%%%%%%%%%%%%%%%%%%%%%%%%%%%%%%%%%%%

\section*{Acknowledgements} \small 

We thank Carlangelo Liverani for motivating us to consider this problem (by encouraging us to prove what we now know to be impossible) and for many helpful discussions. 
We are grateful to Viviane Baladi, Roberto Castorrini, Marco Lenci, Stefano Luzzatto, David Seifert and Daniel Smania for helpful comments, corrections and suggestions. We also would like to thank the anonymous referees for correcting some mistakes and pointing out relevant missing references.
Partially supported by PRIN Grant ``Regular and stochastic behaviour in dynamical systems" (PRIN 2017S35EHN) and MIUR Excellence Department Project awarded to the Department of Mathematics, University of Rome Tor Vergata, CUP E83C18000100006.

\section*{Statements and Declarations} \small 

The authors have no relevant financial or non-financial interests to disclose.
Data sharing not applicable to this article as no datasets were generated or analysed during the current study.
Also, no llamas were involved in the preparation of this manuscript.

\normalsize

%%%%%%%%%%%%%%%%%%%%%%%%%%%%%%%%%%%%%%%%%%

\end{document}